\documentclass[10pt,reqno]{amsart}

\usepackage[margin=1in]{geometry}
\usepackage{amsfonts,amsmath,amssymb,enumerate, amscd}    
\usepackage{amsthm}    
\usepackage{mathrsfs}  
\usepackage{setspace}
\usepackage{stmaryrd}
\usepackage{tikz}
\usetikzlibrary{shapes}

\onehalfspace
\numberwithin{equation}{section}
\theoremstyle{plain}
\newtheorem{thm}{Theorem}[section]
\newtheorem{lem}[thm]{Lemma}
\newtheorem{prop}[thm]{Proposition}
\newtheorem{cor}[thm]{Corollary}
\theoremstyle{definition}
\newtheorem{defn}[thm]{Definition}
\newtheorem{defprop}[thm]{Definition-Proposition}

\theoremstyle{remark}
\newtheorem{rem}[thm]{Remark}

\newtheorem*{rem*}{Remark}

\newcommand{\be}{\begin{equation}}    
\newcommand{\ee}{\end{equation}}    
\newcommand{\beu}{\begin{equation*}}    
\newcommand{\eeu}{\end{equation*}}    
\newcommand{\bea}{\begin{eqnarray}}    
\newcommand{\eea}{\end{eqnarray}}    
\newcommand{\beaa}{\begin{eqnarray*}}    
\newcommand{\eeaa}{\end{eqnarray*}}    
\newcommand{\bmx}{\begin{pmatrix}}    
\newcommand{\emx}{\end{pmatrix}}

\newcommand{\g}{{\mathfrak g}}    
    
\newcommand{\m}{{\mathfrak m}}

\newcommand{\HH}{{\mathcal H}}
\newcommand{\IC}{{\mathbf{IC}}^{\cdot}}

\newcommand{\nn}{\nonumber}

\newcommand{\rank}{{\rm rank}}

\newcommand{\Z}{{\mathbb Z}}
\newcommand{\N}{{\mathbb N}}
\newcommand{\C}{{\mathbb C}}

\renewcommand{\P}{{\bf P}}
\newcommand{\Q}{{\mathbb Q}}    
\newcommand{\R}{{\mathbb R}}

\newcommand{\id}{{\mathrm{id}}}

\newcommand{\uq}{{\mathrm{U}_q}}
\newcommand{\ueps}{{\mathrm{U}_\varepsilon}}

\newcommand{\uqslt}{{\uq(\widehat{\mathfrak{sl}}_2})}
\newcommand{\uepsslt}{{\ueps(\widehat{\mathfrak{sl}}_2})}
\newcommand{\uqgh}{{\uq(\widehat\g)}}
\newcommand{\uqZlg}{{\mathrm{U}_q^{\mathbb Z}(\mathrm L\g)}}
\newcommand{\uqlg}{{\uq(\mathrm L\g)}}
\newcommand{\uepslg}{{\ueps(\mathrm L\g)}}
\newcommand{\uqg}{{\uq(\g)}}

\newcommand{\uepsh}[2]{{\ueps(\widehat{\mathfrak{#1}}_{#2})}}

\newcommand{\Cx}{\mathbb C^*}

\newcommand{\btp}{\begin{tikzpicture}[baseline=0pt,scale=0.9,line width=0.25pt]}    
\newcommand{\etp}{\end{tikzpicture}}

\newcommand{\range}[2]{\llbracket #1,#2 \rrbracket}

\newcommand{\atp}[1]{}

\newcommand{\ie}{i.e. }
\newcommand{\eg}{e.g. }

\newcommand{\M}{{\mathfrak M}}
\newcommand{\LL}{{\mathfrak L}}

\newcommand{\V}{{\mathcal V}}
\newcommand{\Vtilde}{\widetilde{\mathcal V}}
\newcommand{\Vtildeo}{\widetilde{\mathcal V}^\circ}
\newcommand{\Vhat}{\widehat{\mathcal V}}
\newcommand{\W}{{\mathcal W}}
\newcommand{\Hom}{\mathrm{Hom}}
\newcommand{\out}{\mathrm{out}}
\newcommand{\inn}{\mathrm{in}}
\newcommand{\x}[2]{{\bf x}^{#1}_{#2}}
\newcommand{\kk}[2]{{\bf k}^{#1}_{#2}}

\newcommand{\F}{\mathcal F}
\newcommand{\ZZ}{\mathfrak Z}
\newcommand{\Sp}{\mathrm{Sp}}

\newcommand{\GrMod}[1]{\mathsf{Mod}(#1)}

\newcommand{\ev}{\mathrm{ev}}
\newcommand{\ch}{\mathrm{ch}}
\newcommand{\Td}{\mathrm{Td}}
\newcommand{\RR}{\mathrm{RR}_\bullet}

\newcommand{\tens}[1]{\underset{#1}{\otimes}}
\newcommand{\Mhat}{\widehat{\M}}

\newcommand{\Mtilde}{\widetilde{\M}}
\newcommand{\Mtildeo}{\widetilde{\M}^\circ}
\newcommand{\usl}[1]{\mathrm{U}({\mathfrak{sl}}_{#1})}

\date{February 2014}

\author{R. Zegers}
\address{\vspace{-.20cm}
\hspace{-.20cm}
Univ Paris-Sud, Laboratoire de Physique Th\'eorique, UMR8627, Orsay, F-91405;}
\address{
CNRS, Orsay, F-91405.}

\email{
robin.zegers@th.u-psud.fr}

\begin{document} 
{\flushright{\small{LPT-Orsay-14-09
}}\\}

\vskip.5cm
\title[$q,t$-characters and $\ell$-weight spaces]{$q,t$-characters and the structure of the $\ell$-weight spaces of standard modules over simply laced quantum affine algebras
}

\begin{abstract}
We establish, for all simply laced types, a $q,t$-character formula, first conjectured by Nakajima. It relates, on one hand, the structure of the $\ell$-weight spaces of standard modules regarded as modules over the Heisenberg subalgebra of some quantum affine algebra and, on the other hand, the $t$-dependence of their $q,t$-characters, as originally defined in terms of the Poincar\'e polynomials of certain Lagrangian subvarieties in quiver varieties of the corresponding simply laced type. Our proof is essentially geometric and generalizes to arbitrary simply laced types earlier representation theoretical results for standard $\uqslt$-modules.
\end{abstract}

\maketitle

\section{Introduction}
Let $\g$ be a simple finite dimensional Lie algebra and $\widehat{\g}$ be the corresponding (untwisted) affine Ka\v c-Moody algebra. Viewing $\uqgh$ as the \emph{quantum affinization} of $\uqg$ provides a triangular decomposition $(\mathrm{U}^-_q(\widehat \g), \mathrm{U}^0_q(\widehat\g), \mathrm{U}^+_q(\widehat\g))$ with respect to which every finite dimensional $\uqgh$-modules of type 1 turns out to be highest weight in a generalized sense, allowing one to classify all the finite dimensional simple $\uqgh$-modules of type 1, in terms of their so-called highest \emph{$\ell$-weights}, \cite{CPsl2,CPbook}. The latter are encoded in a $\rank(\g)$-tuple of monic polynomials ${\bf P} \in \C[X]_1^{\rank(\g)}$ referred to as \emph{Drinfel'd polynomials}. Indeed, every finite dimensional type 1 $\uqgh$-module admits a direct decomposition into finite dimensional $\ell$-weight spaces whose associated $\ell$-weights and dimensions are fully encoded in a generalization to the quantum affine context of the classical Lie theoretical characters known as \emph{$q$-characters} \cite{FR} -- see also \cite{Knight}. An algorithm due to Frenkel and Mukhin computes these $q$-characters for \emph{minuscule} $\uqgh$-modules, that is for any module with a single dominant $\ell$-weight \cite{FM}. In contrast with what the classical representation theory of Lie algebras might suggest, quantum affine algebras admit many minuscule modules beyond their fundamental modules and, as a consequence, $q$-characters now constitute some of the most efficient tools in the study of finite dimensional modules over quantum affine algebras. An important gap in our understanding of these objects remains however, which $q$-characters alone cannot fill. This gap is related to the existence of \emph{thick} $\uqgh$-modules, by which we mean modules such that some of their $\ell$-weight spaces have dimension strictly greater than one. The Heisenberg subalgebra $\mathrm{U}^0_q(\widehat\g)$ of $\uqgh$, the quantum affine analogue of the Cartan subalgebra of $\g$, is not necessarily realized semisimply on thick $\uqgh$-modules and the actual structure of the latter as $\mathrm{U}^0_q(\widehat\g)$-modules is unkown in general -- see \cite{YZ} though, for a discussion in the case of standard $\uqslt$-modules. It is this structure that we investigate in the present paper, for every simply laced $\g$.

In \cite{Nakajima}, Nakajima gave a K-theoretical construction of quantum affine algebras of simply laced type and of their standard and simple modules. His constructions rely on quiver varieties and generalize previous similar constructions by Ginzburg and Vasserot in type $\mathfrak a$, based on Springer resolutions of partial flag varieties \cite{GV, Vass}. In that context, it is natural to define combinatorial objects refining $q$-characters, the so-called \emph{$q,t$-characters}. Following \cite{NakqtConj, NakqtAxiom}, we define the $q,t$-character of any given standard module $M({\bf P})$, ${\bf P} \in \C[X]_1^{\rank(\g)}$, by
\be \chi_{q,t}(M({\bf P})) = \sum_\rho \mathscr P(\mathfrak L_\rho, t) m_\rho \, ,\ee
where the sum runs over the connected components $\mathfrak L_\rho$ of the generalized Springer fibre $\mathfrak L$ -- see section \ref{sec:gradedquiver} for definitions --, $\mathscr P(\mathfrak L_\rho, t)$ is the Poincar\'e polynomial of $\mathfrak L_\rho$, \ie
\be\mathscr P(\mathfrak L_\rho, t) = \sum_{i=0}^{\dim_{\mathbb R} \mathfrak L_\rho} t^i \dim H_i (\mathfrak L_\rho, \mathbb C)\, ,\ee
and $m_\rho$ is a monomial in the $q$-character $\chi_{q, 1}(M({\bf P}))$, associated with a given $\ell$-weight of $M({\bf P})$. Perhaps the most impressive application of the theory of $q, t$-characters is a Kazhdan-Lusztig type result allowing one to compute the coefficients involved in the change of basis from standard to simple $\uqgh$-modules in the Grothendieck ring of the category of finite dimensional $\uqgh$-modules, \cite{NakqtConj, NakqtAxiom}. As a corollary, Nakajima gave a modified version of the Frenkel-Mukhin algorithm which virtually allows one to determine the $q,t$-character of any finite dimensional simple $\uqgh$-module \cite{NakqtConj}.

Independently of those powerful results, Nakajima conjectured in \cite{NakqtConj} that the $t$-dependence of $q,t$-characters should be related to the natural Jordan filtration of $\ell$-weight spaces by the Heisenberg subalgebra $\mathrm{U}^0_q(\widehat\g)$ of $\uqgh$. To be more precise, let $(\kk{\pm}{i, \pm m})_{i \in I; m \in \N}$ denote the generators of $\mathrm{U}^0_q(\widehat\g)$ giving rise to the $\ell$-weight space decomposition $V= \bigoplus_\rho V_\rho$ of any finite dimensional type 1 $\ueps({\widehat{\g}})$-modules $V$ with $\varepsilon \in \C^*$ and define, for all $\rho$, the filtration $\{0\} = F_{-1}V_\rho \subseteq F_0V_\rho \subseteq F_1 V_\rho \subseteq \dots $, where, for all $k \in \N$,
\be\label{jordanfiltr} F_{k} V_\rho := \bigcap_{\substack{i \in I\\m \in \N}} \ker \left (\kk{\pm}{i, \pm m} - k_{i, \pm m}^\pm (\rho) \, \id \right )^{k+1}\cap V_\rho\, .\ee
$F_\bullet V_\rho$ clearly stabilizes at some $n_\rho \in \N$ that we shall refer to as the \emph{length} of the Jordan filtration. Then the purpose of the present paper is to establish the following
\begin{thm}\label{thm:main}
Let $\varepsilon \in \C$ be transcendental over $\Q$ and $M({\bf P})$ be a standard module. Letting $M({\bf P})= \bigoplus_\rho M({ \bf P})_\rho$ be its $\ell$-weight space decomposition and denoting by $n_\rho \in \N$ the length of the Jordan filtration $F_\bullet M({\bf P})_\rho$, we have
\be \label{qtfiltr} \chi_{\varepsilon,t}(M({\bf P})) = \sum_\rho \sum_{k =0}^{n_\rho} t^{2\sigma_\rho(k)} \dim(F_kM({\bf P})_\rho /F_{k-1}M({\bf P})_\rho) m_\rho\, ,\ee
where $\sigma_{\rho}$ can be chosen as the permutation of $\range{0}{n_\rho}$ defined by $\sigma_\rho(k) =\lfloor n_\rho/2 \rfloor - k/2$ for even $k$ and $\sigma_\rho(k) = \lfloor n_\rho/2 \rfloor + \lceil k/2 \rceil$ for odd $k$.
\end{thm}

Note that, for a given ${\bf P} \in \mathbb C[X]_1^{\rank (\g)}$, the formula (\ref{qtfiltr}) indeed holds for all but a finite number of algebraic values of $\varepsilon \in \C$. In the case $\g = \mathfrak a_1$, this theorem specializes to a result of \cite{YZ}. 
Unlike the combinatorial/representation theoretical proofs in \cite{YZ}, our proof of theorem \ref{thm:main} is essentially geometric. It relies on the ampleness of a family of tautological locally free sheaves over quiver varieties involved in the K-theoretical realization of the Heisenberg subalgebra $\mathrm{U}^0_q(\widehat\g)$ of $\uqgh$ and on the hard Lefschetz theorem for their first Chern class. The latter allows one to endow each $\ell$-weight space $M({ \bf P})_\rho$ with a canonical $\mathrm{U}(\mathfrak{sl}_2)$-module structure, closely related to the Lefschetz primitive decomposition of the cohomology of the associated connected component of the generalized Springer fibre. Theorem \ref{thm:main} then follows by relating this $\mathrm{U}(\mathfrak{sl}_2)$-module structure to the Jordan filtration $F_\bullet M({\bf P})_\rho$ and by regarding the Poincar\'e polynomial $\mathscr P(\mathfrak L_\rho, t)$ as the (Lie theoretical) character of the $\mathrm{U}(\mathfrak{sl}_2)$-module $M({ \bf P})_\rho$ so obtained. It is worth emphasizing that an approach in the spirit of \cite{YZ} yields a formidable combinatorial problem, which is why we chose to recast it in geometric terms, in order to benefit from powerful results in geometry such as the hard Lefschetz theorem. Moreover, this approach provides a geometric interpretation of an equally crucial $\mathrm{U}(\mathfrak{sl}_2)$-module structure of the $\ell$-weight spaces appearing in \cite{YZ}, where its construction though seemed rather \emph{ad hoc}.

The paper is organized as follows. We recall the definition of quantum loop algebras in terms of Drinfel'd current generators in section \ref{sec:QLA}. In section \ref{sec:quiver}, we review the construction of Nakajima's quiver varieties and the associated K-theoretical realization of quantum affine algebras of simply laced type. Readers familiar with Nakajima's work can skip to section \ref{sec:gradedquiver} where, after reviewing well known material about graded quiver varieties, we prove the ampleness of certain tautological sheaves over graded quiver varieties. This new result is crucial to section \ref{sec:proof}, where we eventually establish theorem \ref{thm:main}. The (geometric) definition of standard modules over quantum affine algebras and the geometric realization of its $\mathrm{U}^0_q(\widehat\g)$-module structure can be found in section \ref{sec:stdmodules}.  The latter essentially follows \cite{Nakajima} except  that, for the sake of completeness and given its key role in proving theorem \ref{thm:main}, we slightly elaborate -- see in particular lemmas \ref{lem:ra}, \ref{lem:RR} and \ref{lem:star} -- on the proof of proposition \ref{prop:kaction} which provides a realization of the $\mathrm{U}^0_q(\widehat\g)$-module structure of $\ell$-weight spaces in terms of the Chern characters of the above mentioned tautological sheaves. In section 7, we give examples of applications of theorem \ref{thm:main}.

\subsection*{Notations and conventions} Troughout this paper, we adopt the following notations
\be [n]_q:= \frac{q^{n}-q^{-n}}{q-q^{-1}} \, , \quad \qquad [n]_q!:= \prod_{m=1}^n [m]_q\, , \quad \qquad {n \choose p}_{\!\! q} := \frac{[n]_q!}{[p]_q! [n-p]_q!} \, , \ee
for every $n, p \in \mathbb N$ and every $q \in \mathbb C^*$. For every integer $n \in \mathbb N$, we let $S_n$ denote the symmetric group of bijections on $\range1 n$ and every $\sigma \in S_n$ acts as $(1, \dots , n) \mapsto (\sigma(1), \dots ,\sigma(n))$.

Let $G$ be a complex algebraic linear group. With every quasi-projective $G$-variety $X$, we associate the Grothendieck group $K^G(X)$ (resp. $K_G(X)$) of the category $\mathsf{Coh}^G(X)$ (resp. $\mathsf{Loc}^G(X)$) of $G$-equivariant coherent (resp. locally free) sheaves of $\mathcal O_X$-modules. We shall systematically identify locally free sheaves -- and sheaves of sections of algebraic vector bundles -- with the isomorphism classes they define in $K_G(X)$ and, whenever $X$ is non-singular, the elements of $K_G(X)$ with their images in $K^G(X)$ under the isomorphism $K_G(X) \cong K^G(X)$. Let $R(G):= K^G(\mathrm{point})$ be the representation ring of $G$. $K^G(X)$ is an $R(G)$-module. For convenience and provided it does not give rise to confusion, we shall always denote with the same symbol a sheaf and its class in the corresponding Grothendieck group. The class of the structure sheaf of $X$ will thus be simply denoted by $\mathcal O_X$ or even $1$, when the context permits. We shall similarly denote by $f_*$, $f^*$ and $\otimes$ the derived functors $Rf_*$, $Lf^*$ and $\stackrel{L}{\otimes}$ respectively. For any given sheaf $\mathcal F$, we shall denote by $\mathcal F^\vee$ its dual sheaf. $\mathrm{Pic}^G(X)$ will denote the Picard group of isomorphism classes of $G$-equivariant invertible sheaves of $\mathcal O_X$-modules. Finally, for every $G$-equivariant vector bundle $\mathcal V$, define
\be \det \mathcal V := {\bigwedge}^{\rank \mathcal V} \mathcal V \in \mathrm{Pic}^G(X) \qquad \mathrm{and} \qquad  {\bigwedge}_z \mathcal V := \sum_{k=0}^{\rank(\mathcal V)} z^k {\bigwedge}^k \mathcal V  \in K^G(X)[z]\, . \ee
\newpage

\section{Quantum loop algebras}
\label{sec:QLA}
Let $\g$ be a simple, simply laced complex Lie algebra and let $I:= \range{1}{\rank \,\,\g}$. Denote by $\langle, \rangle$ the invariant inner product on $\g$, normalized in such a way that the square length of the maximal root be $2$. Let $(\alpha_i)_{i \in I}$ and $(\omega_i)_{i\in I}$ respectively denote the sets of simple roots and of fundamental weights of $\g$ and set $Q:= \bigoplus_{i \in I} \mathbb Z \alpha_i$, $Q^+:= \bigoplus_{i \in I} \mathbb N \alpha_i$, $P:= \bigoplus_{i \in I} \mathbb Z \omega_i$ and $P^+:= \bigoplus_{i \in I} \mathbb N \omega_i$. The Cartan matrix ${\bf C}=(c_{ij})_{i,j \in I}$ is given by $c_{ij} = \langle \alpha_i , \alpha_j\rangle$. The \emph{quantum loop algebra} associated with $\g$ is the $\mathbb C(q)$-algebra $\uqlg$ generated by $(\x{\pm}{i,m})_{i \in I, m \in \mathbb Z}$ and $(\kk{\pm}{i,m})_{i \in I, m \in \pm \mathbb N}$, subject to the defining relations
\be \kk{+}{i,0}\kk{-}{i,0} = \kk{-}{i,0}\kk{+}{i,0} = 1\, , \qquad [\kk{(\pm)_1}{i}(z_1), \kk{(\pm)_2}{j}(z_2)] =0 \, ,\ee
\be (z_1 - q^{(\pm)_1 c_{ij}} z_2) \kk{(\pm)_2}j(z_1) \x{(\pm)_1}{i}(z_2) = (q^{(\pm)_1 c_{ij}}z_1 -  z_2) \x{(\pm)_1}{i}(z_2) \kk{(\pm)_2}j(z_1) \, , \ee
\be (z_1 - q^{\pm c_{ij}} z_2) \x{\pm}i(z_1) \x{\pm}{j}(z_2) = (q^{\pm c_{ij}}z_1 -  z_2) \x{\pm}{j}(z_2) \x{\pm}{i}(z_1)\, ,  \ee
\be [\x{+}{i}(z_1), \x{-}{j}(z_2)] = \frac{\delta_{ij}}{q-q^{-1}} \left [ \delta \left (\frac{z_2}{z_1} \right ) \kk{+}{i}(z_2) - \delta \left (\frac{z_1}{z_2} \right ) \kk{-}{i}(z_1) \right ]\, ,\ee
\be \sum_{\sigma \in S_{1-c_{ij}}}\sum_{p=0}^{1-c_{ij}} {{1-c_{ij}}\choose{p}}_{\!\!q} \,\,\x{\pm}{i}(z_{\sigma(1)}) \cdots \x{\pm}{i}(z_{\sigma(p)}) \x{\pm}{j}(z) \x{\pm}{i}(z_{\sigma(p+1)}) \cdots \x{\pm}{i}(z_{\sigma(1-c_{ij})}) =0 \, ,\ee
written here in terms of the formal series
\be \delta(z) := \sum_{m \in \mathbb Z} z^m\, ,\ee
\be \x{\pm}{i}(z) :=  \sum_{m \in \mathbb Z} \x{\pm}{i,m} z^{-m}\, ,\ee
\be \kk{\pm}{i}(z) := \sum_{m \in \mathbb N} \kk{\pm}{i,\pm m} z^{\mp m}\, .\ee
In order to specialize the deformation parameter $q$ to some value $\varepsilon \in \mathbb C^*$, we also introduce the $\mathbb C[q, q^{-1}]$-subalgebra $\uqZlg$ of $\uqlg$, generated by the
\be \x{\pm}{i,m, n} := \frac{(\x{\pm}{i,m})^n}{[n]_q!} \ee
for all $i \in I$, $m \in \mathbb Z$ and $n \in \mathbb N^*$, together with the coefficients in the $z^{\pm 1}$ expansion of
\be \exp \left ( - \sum_{m \in \mathbb N^*} \frac{{\bf h}_{i,\pm m}}{[m]_q} z^{\mp m}\right )\ee
where, for all $i \in I$ and $m \in \mathbb Z^*$, the ${\bf h}_{i,m}$ are defined by
\be \kk{\pm}{i} := \kk{\pm}{i,0} \exp \left (\pm (q-q^{-1}) \sum_{m \in \mathbb N^*} {\bf h}_{i, \pm m } z^{\mp m}\right ) \, .\ee
$\uqZlg$ constitutes an integral form of $\uqlg$, \ie $\uqlg \cong \uqZlg \otimes_{\C[q, q^{-1}]} \C(q)$. For every $\varepsilon \in \C^*$, we thus define $\uepslg$, the \emph{specialization} at $\varepsilon$ of $\uqlg$, as $\uqZlg \otimes_{\C[q,q^{-1}]} \C$, where $\C$ is a $\C[q,q^{-1}]$-module through $q \mapsto \varepsilon$.

It is well known that $\uqlg$ is a quotient of $\uqgh$ and that every finite dimensional simple $\uqgh$-module is obtained by twisting some $\uqlg$-modules by some $\C(q)$-algebra automorphism of $\uqgh$. We therefore restrict our attention to $\uqlg$ for the rest of the paper.

\section{Quiver varieties}
\label{sec:quiver}
General quiver varieties pertain to a class of \emph{Geometric Invariant Theory} (GIT) quotients studied by King in \cite{King}. We first describe the construction of these GIT quotients in general. We then give a short review of a particular class of quiver varieties introduced by Nakajima, \cite{Nakajima}, in order to study the representation theory of quantum affine algebras.

\subsection{Algebraic quotients} \label{GIT} Suppose that $X = \mathrm{Spec} \, A$ is an affine algebraic variety and that $G$ is a reductive algebraic group acting on $X$. Let $X/\!/ G := \mathrm{Spec} \, A^G$ be the affine quotient of $X$ by $G$. With every character $\theta \in \Hom(G, \C^*)$ we associate a $G$-linearized line bundle $\mathscr L_\theta:= X \times \C_\theta \in \mathrm{Pic}^G(X)$, where $\C_\theta$ is the one dimensional $G$-module with $G$-action given by $ G \times \C_\theta \ni (g, z) \mapsto g.z:= \theta(g) z \in \C_\theta$. Since $G$ is reductive the graded algebra of invariant sections
\be A^\theta_\bullet :=\bigoplus_{n \in \N} H^0(X, {\mathscr L}_\theta^{\otimes n})^{G}\ee
is finitely generated and we can define the algebraic quotient $X/\!/_\theta \, G:= \mathrm{Proj} \,\, A^\theta_\bullet$. The latter is quasi-projective, the inclusion $A^G = A^\theta_0 \hookrightarrow A^\theta_\bullet$ inducing a projective morphism
\be \pi_\theta : X/\!/_\theta \, G \rightarrow X/\!/  G\, .\ee

GIT provides a more geometric description of $X/\!/_\theta \, G$ relying on the following notions of (semi)stability,
\begin{defn}
Let $\theta \in \Hom(G, \C^*)$ be a character.  We say that a point $x \in X$ is \emph{$\theta$-semistable} iff there exist $n \in \N^*$ and $s \in A^\theta_n$ such that $s(x) \neq 0$. We shall denote by $X^{\theta}$ the set of $\theta$-semistable points in $X$. We say that a point $x \in X$ is \emph{$\theta$-stable} iff it is $\theta$-semistable and, in addition, its stabilizer $G_x$ is finite and its orbit $G.x$ is closed in $X^\theta$. We shall denote by $X^{\underline{\theta}}$ the set of $\theta$-stable points in $X$.
\end{defn}
We define the GIT equivalence $\sim_G$ over $X^{\theta}$ by setting, for every $x, y \in X^{\theta}$, $x \sim_G y$ iff $\overline{G.x} \cap \overline{G.y} \cap X^\theta \neq \emptyset$, where a bar denotes the Zariski closure.
The relevance of this equivalence relation is clear from the following
\begin{prop}
\label{prop:GIT-Quot}
$X/\!/_\theta \, G$ is isomorphic to the quotient $X^\theta/ \sim_G$.
\end{prop}
Note that there always exists a geometric quotient of $X^{\underline{\theta}}$ by $G$ -- see \eg \cite{Mumford}. That geometric quotient constitutes a Zariski-open subset of $X/\!/_\theta \, G$. The following is a direct consequence of Luna's \'etale slice theorem \cite{Luna}.
\begin{prop}
\label{prop:canampleline}
Suppose that
\begin{itemize}
\item[-] $X^{\underline{\theta}} = X^\theta \neq \emptyset$;
\item[-] $X^\theta$ is non-singular;
\item[-] and the action of $G$ on $X^\theta$ is free.
\end{itemize}
Then, the following assertions hold:
\begin{enumerate}
 \item[i.] $X^\theta$ is a principal $G$-bundle over $X/\!/_\theta \, G$;
\item[ii.] $X/\!/_\theta \, G$ is non-singular and isomorphic to $X^\theta /G$;
\item[iii.] the line bundle $\mathscr L_\theta$ descends to a line bundle
\be \mathcal O_{X/\!/_\theta \, G}(1) := X^\theta \times_G \left (\mathscr L_\theta |_{X^\theta} \right )\ee
which is ample relative to $\pi_\theta : X/\!/_\theta \, G \rightarrow X/\!/G$ and therefore ample over $X/\!/_\theta \, G$ since $X/\!/G$ is affine.
\end{enumerate}
\end{prop}
\begin{rem}
In general, the principal $G$-bundle structure mentioned in the above proposition involves local triviality in the \'etale topology. However, when $G$ is the general linear group $GL(n, \C)$ as will be the case here -- or more generally when $G$ is \emph{special} by definition of that term, \cite{Serre} -- every principal $G$-bundle is indeed locally trivial in the Zariski topology.
\end{rem}

\subsection{Definitions}
\label{quivers} We now turn to the definition of quiver varieties, following \cite{Nakajima}. Remember that $I:=\range{1}{\rank(\g)}$. Let $(I, E)$ be the finite graph whose vertices, labelled by $I$, are related by $2 \delta_{ij}- c_{ij}$ edges in $E$ for every pair of vertices $i, j \in I$. In the cases of interest in this paper, these graphs are nothing but Dynkin diagrams of type $\mathfrak a$, $\mathfrak d$, or $\mathfrak e$. Let $H$ denote the set of pairs consisting of an edge together with its orientation. For each $h \in H$, let $\inn(h)$ (resp. $\out(h)$) denote the ingoing (resp. outgoing) vertex of $h$ and $\bar h$ denote the same edge as $h$ with opposite orientation. Let $\varepsilon : H \rightarrow \mathbb C^*$ be such that $\varepsilon (h ) + \varepsilon (\bar h) = 0$ for all $h \in H$. With every pair $(V, W)$ of $I$-graded finite dimensional $\mathbb C$-vector spaces, we associate the following $\mathbb C$-vector spaces
\be L(V, W) := \bigoplus_{k \in I} \Hom (V_k, W_k)\, ,\label{eq:L}\ee
\be E(V, W) := \bigoplus_{h \in H} \Hom (V_{\out(h)}, W_{\inn(h)}) \, ,\label{eq:E}\ee
\be M(V, W) := E(V, V) \oplus L(W, V) \oplus L(V, W) \, .\label{eq:M}\ee
Elements of $M(V, W)$ will be conveniently written as triples $(B, i, j)$ with $B=(B_h)_{h \in H} \in E(V, V)$, $i=(i_k)_{k \in I} \in L(W, V)$ and $j =  ( j_k)_{k \in I} \in L(V, W)$. The above defined vector space $M(V, W)$ admits a group action of $G(V) := \prod_{k \in I} GL(V_k)$ defined by
\be (B,i,j) \mapsto g \cdot (B,i,j) := (g\cdot B, g\cdot i, g \cdot j) \, , \label{eq:Gv}\ee
for all $(B,i,j) \in M(V, W)$ and all $g =(g_k)_{k \in I}\in G(V)$, where 
\be g \cdot B := (g_{\inn(h)} \circ B_h \circ g_{\out(h)}^{-1})_{h \in H} \qquad g \cdot i := (g_k \circ i_k)_{k \in I} \qquad g\cdot j:= (j_k \circ g^{-1}_k)_{k \in I}\, .\ee
We shall denote by $[B,i,j]$ the $G(V)$-orbit of $(B,i,j) \in M(V, W)$. Define furthermore the \emph{momentum map}
\bea \mu : M(V, W) &\rightarrow& L(V,V) \nn\\ 
(B,i,j) &\mapsto& \mu(B,i,j) := \left (\sum_{\substack{h \in H\\ \inn(h) = k}} \varepsilon(h) B_h \circ B_{\bar h} + i_k \circ j_k \right )_{k \in I} \, . \label{eq:defmu}
\eea
Let $\mu^{-1}(0) \subseteq M(V, W)$ be the affine algebraic variety defined as the zero set of $\mu$ and denote by $A(\mu^{-1}(0))$ its coordinate ring.

Define a character $\chi : G(V) \rightarrow \mathbb C^*$ by setting $\chi(g) = \prod_{k \in I} \det g_k^{-1}$ for all $g=(g_k)_{k \in I} \in G(V)$. The considerations of the previous section allow us to make the following
\begin{defn}
For every pair $(V, W)$ of finite dimensional $I$-graded $\mathbb C$-vector spaces, let
\be \M_0 (V, W):= \mu^{-1}(0) /\!/ G(V) = \mathrm{Spec} \, A(\mu^{-1}(0))^{G(V)}\, ,\ee
and
\be \M (V, W):= \mu^{-1}(0)/\!/_\chi \, G(V) = \mathrm{Proj} \, A_\bullet^\chi \, .\ee
From now on, $\M(V, W)$ -- or, to be more precise, its subspace of closed points --  will be referred to as a \emph{quiver variety}.
\begin{rem}
In order to make the relation with representation theory more transparent, it is worth noting that one could equivalently parametrize quiver varieties and all related objects in terms of the root and weight associated with the respective dimension vectors $\dim V = (\dim V_k)_{k \in I}$ and $\dim W = (\dim W_k)_{k \in I}$ by setting 
\be \alpha(V): = \sum_{k\in I} \dim(V_k) \, \alpha_k \in Q^+ \qquad \qquad \lambda(W):= \sum_{k \in I} \dim (W_k) \, \omega_k \in P^+\, .\ee
\end{rem}
\end{defn}

As in the previous section, $\M(V, W)$ is quasi-projective. Indeed, we have
\begin{prop}
\label{Prop:projmorph}
For every pair $(V, W)$ of $I$-graded finite dimensional $\mathbb C$-vector spaces, there exists a projective morphism $\pi : \M(V, W) \rightarrow \M_0(V, W)$.
\end{prop}

For every $x \in \M_0(V, W)$, define
\be \LL(V,W)_x:= \pi^{-1}(x)\, ,\ee
the fibre of $\pi$ at $x$. We shall denote by $\LL(V, W)$, instead of $\LL(V,  W)_0$, the fibre at $0 \in \M_0(V, W)$. By construction, $\LL(V, W)$ is a projective subvariety of $\M(V, W)$. It is known that $ \dim \M(V, W) = 2 \dim \LL(V, W)$ and that, indeed, $\LL(V,W)$ is a Lagrangian subvariety of $\M(V, W)$, when the latter is endowed with its natural symplectic structure, \cite{Nakajima}.

Turning our attention to $\M_0(V, W)$, we define $\M_0^{\mathrm{reg}}(V, W)$ as the (possibly empty) set of points $[B,i,j] \in \M_0(V, W)$ such that $(B,i,j) \in \mu^{-1}(0)$ has the trivial stabilizer in $G(V)$. In \cite{NakKM}, it is proven that every $[B,i,j] \in \M_0^{\mathrm{reg}}(V, W)$ is stable, that $\pi$ induces an isomorphism $\pi^{-1}(\M_0^{\mathrm{reg}}(V, W)) \cong \M_0^{\mathrm{reg}}(V, W)$ and, furthermore, that
\begin{prop}
\label{prop:regular}
For a graph $(I,E)$ of type $\mathfrak a$, $\mathfrak d$, or $\mathfrak e$, and every pair of $I$-graded vector spaces $(V, W)$
\begin{itemize}
\item[-] we have $\M_0(V, W) = \bigcup_{[V']} \M_0^{\mathrm{reg}}(V', W)$, where the union runs over equivalence classes of $I$-graded vector spaces $[V']$ such that $\alpha(V) - \alpha(V') \in Q^+$;
\item[-] and if $\M_0^{\mathrm{reg}}(V, W) \neq \emptyset$ then $\lambda(W) - \alpha(V)$ is dominant.
\end{itemize}
\end{prop}
As a consequence, the union over equivalence classes of $I$-graded vector spaces $\bigcup_{[V]} \M_0(V, W)$ stabilizes at some $[V]$ and we let
\be \M_0(\infty, W) := \bigcup_{[V]} \M_0(V, W)\, . \label{eq:M0infty}\ee

\subsection{Stability}
The set $\mu^{-1}(0)^\chi$ of $\chi$-semistable points admits a useful and more intrinsic characterization through the following
\begin{prop}
\label{prop:stab}
A point $(B, i, j) \in \mu^{-1}(0)$ is $\chi$-semistable iff there exists no non-trivial $B$-invariant $I$-graded subspace of $\ker j$.
\end{prop}
This is proved in \cite{NakKM}, using Hilbert's criterion. Stated in the above form, $\chi$-semistability easily implies that the stabilizer $G(V)_{(B,i,j)}$ of any triple $(B, i,j) \in \mu^{-1}(0)^\chi$ is trivial. It follows that all $G(V)$-orbits are closed in $\mu^{-1}(0)^\chi$ and that, indeed, $\mu^{-1}(0)^{\underline{\chi}} = \mu^{-1}(0)^\chi$. Thus, $\chi$-semistability and $\chi$-stability agree and are merely referred to as \emph{stability} in the context of quiver varieties. Using proposition \ref{prop:stab}, one also establishes that the differential $d\mu:M(V, W) \to L(V, V)$ is surjective at every triple $(B,i,j) \in \mu^{-1}(0)^\chi$ and hence that $\mu^{-1}(0)^\chi$ is non-singular \cite{NakKM}. We are thus under the premises of proposition \ref{prop:canampleline} and its conclusions follow, namely \cite{Nakajima}
\begin{prop}
\label{prop:principalbundle}
For every pair $(V, W)$ of $I$-graded finite dimensional $\mathbb C$-vector spaces, the following hold:
\begin{enumerate}
\label{prop:Mquot}
\item[i.] $\mu^{-1}(0)^\chi$ has the structure of a principal $G(V)$-bundle over $\mu^{-1}(0)^\chi/G(V)$;
\item[ii.] $\M (V, W)$ is isomorphic to the $\langle \alpha(V), 2\lambda(W) - \alpha(V)\rangle$-dimensional non-singular quasi-projective variety $\mu^{-1}(0)^\chi/G(V)$;
\item[iii.] the line bundle $\mathscr L_\chi$ descends to an ample line bundle $\mathcal O_{\M(V, W)}(1) := \mu^{-1}(0)^\chi \times_{G(V)} (\mathscr L_\chi|_{\mu^{-1}(0)^\chi})$ on $\M(V, W)$. 
\end{enumerate}
\end{prop}

\subsection{$G(W) \times \mathbb C^*$-action}
Let $G(W):=\prod_{k\in I} GL(W_k)$. $M(V, W)$ admits a $G(W) \times \mathbb C^*$-action defined by
\be G(W) \times \mathbb C^* \times M(V, W) \ni (g, t, (B,i,j))  \mapsto (g,t) *(B,i,j):= (t\, B, t\, i\circ g^{-1}, t \, g \circ j) \, .\ee
Since it commutes with the $G(V)$-action, as defined in equation \ref{eq:Gv}, and leaves the equation $\mu(B,i,j)=0$ invariant, this action descends to the quotients $\M(V, W)$ and $\M_0(V, W)$.

With every integer $m \in \mathbb Z$, we associate a one dimensional $\mathbb C^*$-module $q^m$ by setting~\footnote{In \cite{Nakajima}, a more general $\mathbb C^*$-action was proposed so as to make the tautological bundles defined in the next section $G(W)\times \mathbb C^*$-equivariant for quiver varieties of arbitrary type. That $\mathbb C^*$-action nonetheless coincides with the one used here in type $\mathfrak a$, $\mathfrak d$, or $\mathfrak e$, which are the cases of interest in the present paper.}
\be  \mathbb C^* \times q^m \ni (t,z) \mapsto t^m \, z \in q^m\, .\ee
Given any $\mathbb C^*$-module $V$, we shall write $q^m V$ as a shorthand for $q^m \otimes V$.

\subsection{Vector bundles on $\M(V, W)$}
In view of proposition \ref{prop:principalbundle}, we introduce the \emph{tautological} $I$-graded vector bundle
\be \mathcal V:= \mu^{-1}(0)^\chi \times_{G(V)} V \, ,\ee
as the associated vector bundle with fibre $V$ of the principal $G(V)$-bundle $\mu^{-1}(0)^\chi \rightarrow \M(V, W)$. Similarly, let $\mathcal W$ be the \emph{trivial} vector bundle on $\M(V, W)$ with fibre $W$. For all $k \in I$, the $k$-th components of $\mathcal V$ and $\mathcal W$ will be denoted $\mathcal V_k$ and $\mathcal W_k$ respectively. Equations (\ref{eq:L}) and (\ref{eq:E}) then define, fibrewise, vector bundles on $\M(V, W)$ that we shall denote $E(\mathcal V, \mathcal V)$, $L(\mathcal W, \mathcal V)$ and $L(\mathcal V,\mathcal W)$, of which $B$, $i$ and $j$ are now regarded as sections. $\mathcal V$ is then naturally a $\mathbb C^*$-equivariant vector bundle. Letting $G(W)$ act trivially on it, we make it a $G(W)\times \mathbb C^*$-equivariant vector bundle. Similarly, $\mathcal W$ is a $G(W)\times \mathbb C^*$-equivariant vector bundle by taking the natural action of $G(W)$ and the trivial action of $\mathbb C^*$. As a consequence, the vector bundles $E(\mathcal V, \mathcal V)$, $L(\mathcal W, \mathcal V)$ and $L(\mathcal V, \mathcal W)$ are also $G(W)\times \mathbb C^*$-equivariant vector bundles and, correspondingly, $B$, $i$ and $j$ are $G(W)\times \mathbb C^*$-equivariant sections.

For every $m \in \mathbb Z$, let $q^m$ denote the trivial line bundle on $\M(V, W)$ with degree $m$ $\mathbb C^*$-action and consider, for all $k \in I$, the classes
\be \mathcal F_k(V, W) := q^{-1} \mathcal W_k- (1+q^{-2}) \mathcal V_k + q^{-1} \sum_{\substack{h \in H\\ \inn(h)=k}} \mathcal V_{\out (h)}\ee
in the Grothendieck group $K^{G(W) \times \mathbb C^*}(\M(V, W))$ of the abelian category of $G(W) \times \mathbb C^*$-equivariant coherent sheaves of $\mathcal O_{\M(V, W)}$-modules. They are the images in $K^{G(W) \times \mathbb C^*}(\M(V, W))$ of the classes defined, in the Grothendieck group $K^{G(W)\times \C^*}(\M(V, W), \M(V, W))$ of the derived category of $G$-equivariant complexes of algebraic vector bundles over $\M(V, W)$ exact outside $\M(V, W)$, by the complexes \cite{Nakajima}
\be C_k^\bullet(V,W) : q^{-2} \V_k \stackrel{\sigma_k}{\longrightarrow} q^{-1} \W_k \oplus  \bigoplus_{\substack{h \in H\\ \inn(h)=k}} q^{-1}\mathcal V_{\out (h)} \stackrel{\tau_k}{\longrightarrow} \V_i\, ,\ee
where, for all $k \in I$,
\be \sigma_k := \bigoplus_{\substack{h \in H\\ \inn(h) = k}} B_{\bar{h}} \oplus j_k  \qquad \mbox{and} \qquad \tau_k :  \sum_{\substack{h \in H\\ \inn(h) = k}} \varepsilon(h) B_h + i_k\, .\ee

\subsection{The convolution algebra}
\label{Sec:Convolution}
Let $X_1$, $X_2$ and $X_3$ be smooth quasi-projective algebraic varieties and denote by $p_{ab}:X_1 \times X_2 \times X_3 \rightarrow X_a \times X_b$, for all $a< b \in \range{1}{3}$. We consider $G$-stable closed subvarieties $Z_{ab} \subset X_a \times X_b$ such that $p_{13} : p_{12}^{-1}Z_{12}\cap p_{23}^{-1} Z_{23} \rightarrow X_1 \times X_3 $ be proper and we let $Z_{12}\circ Z_{23} := p_{13}(p_{12}^{-1}Z_{12}\cap p_{23}^{-1} Z_{23} )$. We then follow \cite{CG} in making the following
\begin{defn}
The \emph{convolution product} is the map
\bea \star : K^G(Z_{12}) \otimes K^G(Z_{23}) &\rightarrow& K^G(Z_{12} \circ Z_{23}) \nn\\
\mathcal F \otimes \mathcal F' &\mapsto& p_{13*}((p_{12}^* \mathcal F)\otimes (p_{23}^* \mathcal F'))\, .
\eea
\end{defn}

\begin{defn}
For every triple $(V,V', W)$ of $I$-graded $\mathbb C$-vector spaces, we define $\mathfrak Z(V, V', W)$ as the fibered product
\be\mathfrak Z(V, V', W):= \M(V, W) \times_\pi \M(V', W)\, , \ee
regarded as fibered over $\M_0(\infty, W)$ -- see eq. (\ref{eq:M0infty}).
\end{defn}
Let
\be K^{G(W) \times \mathbb C^*}(\mathfrak Z(W)) := \prod_{[V], [V']} K^{G(W) \times \mathbb C^*}(\mathfrak Z(V, V', W)) \, ,\ee
where, the graph $(I, E)$ being of type $\mathfrak a$, $\mathfrak d$, or $\mathfrak e$, the direct product runs over the finite number of isomorphism classes $[V]$ of $I$-graded vector spaces such that $\M(V, W)$ be non-empty.
\begin{prop}
$K^{G(W) \times \mathbb C^*}(\mathfrak Z(W))$, endowed with the above defined convolution product, constitutes a $\mathbb Z[q,q^{-1}]$-algebra.
\end{prop}
\begin{proof}
By definition, $\mathfrak Z(V, V', W)$ is a closed subvariety of  $\M(V, W) \times \M(V', W)$. Furthermore, the map $p_{13}: p_{12}^{-1}(\mathfrak Z(V, V', W))\cap p_{23}^{-1}( \mathfrak Z(V', V'', W)) \rightarrow \M(V, W) \times \M(V'', W)$ is proper and its image $\mathfrak Z(V, V', W) \circ \mathfrak Z(V', V'', W) \subseteq \mathfrak Z(V, V'', W)$. Hence, the convolution product
\be \star :  K^{G(W) \times \mathbb C^*}(\mathfrak Z(V, V', W)) \times K^{G(W) \times \mathbb C^*}(\mathfrak Z(V', V'', W)) \rightarrow K^{G(W) \times \mathbb C^*}(\mathfrak Z(V, V'', W))\ee
is well defined. Observe furthermore that $R(G(W) \times \mathbb C^*)$ is an $R(\mathbb C^*)$-algebra and that $R(\mathbb C^*)$ is isomorphic to $\mathbb Z[q,q^{-1}]$ through $q^m \mapsto L(m)$. Hence we have that $K^{G(W) \times \mathbb C^*}(\mathfrak Z(W))$ is a $\mathbb Z[q,q^{-1}]$-algebra. 
\end{proof}

Set $\mathcal F_i(W):=\bigoplus_{[V]}\mathcal F_i(V, W)$ and let $f_i(W)$ denote the diagonal operator acting on $K^{G(W) \times \mathbb C^*}(\M(V, W))$ by the scalar 
\be\rank \,\, \mathcal F_i(V, W) = \langle \alpha_i , \lambda(W) - \alpha(V) \rangle\, ,\ee 
for every $I$-graded $\mathbb C$-vector space $V$. Let furthermore $H^+ \subset H$ be such that $H^+ \cap\bar H^+ = \emptyset$ and $H=H^+ \cup \bar H^+$. For every $i, j \in I$, denote by $n^+_{ij}$ the number of oriented edges in $H^+$ relating $i$ and $j$ and set $n^-_{ij} := 2 \delta_{ij} - c_{ij}-n^+_{ij}$, where, remember, the $c_{ij}$ denote the entries of the Cartan matrix of $\g$. Clearly $n^+_{ij}= n^-_{ij}$. Define
\be \mathcal F_i^-(V, W):= - \V_i + q^{-1} \sum_{j \in I} n_{ij}^-\V_j \qquad \qquad \mathcal F_i^+(V, W):= q^{-1} \W_i - q^{-1}\V_i +q^{-1} \sum_{j \in I} n^+_{ij} \V_j\ee
and let $f_i^\pm(W)$ denote the diagonal operator acting on $K^{G(W) \times \mathbb C^*}(\M(V, W))$ by the scalar $\rank \,\, \mathcal F_i^\pm(V, W)$. Finally, whenever $V \subset V'$ with $\alpha(V') - \alpha(V) = \alpha_i$ for some $i \in I$, define $\mathfrak C_i^+(V', W) \subset \ZZ(V, V', W)$ to be the set of pairs $((B, i, j), (B', i', j'))$ such that $B'|_V=B$, $i'=i$ and $j'|_V= j$; and, whenever on the contrary $\alpha(V) - \alpha(V') = \alpha_i$, let $\mathfrak C_i^-(V', W):= \omega(\mathfrak C_i^+(V', W) \subset \ZZ(V, V', W)) \subset \ZZ(V, V', W)$, where $\omega$ permutes the factors.
\begin{thm}[\cite{Nakajima}]
\label{thm:Nakajima}
There exists a unique homomorphism of $\mathbb C(q)$-algebras
\be\Phi_{W} : \uqlg \rightarrow K^{G(W)\times \mathbb C^*}(\mathfrak{Z}(W))\otimes_{\mathbb Z[q,q^{-1}]} \mathbb C(q)\ee 
such that, for all $i \in I$ and every $m \in \Z$,
\bea 
\kk{\pm}{i}(z) &\mapsto& \delta_* q^{f_i(W)} {\bigwedge}_{-1/z} \left ((q^{-1}-q) \mathcal F_i(W)\right )\, , \\
\x{\pm}{i,m} &\mapsto& \sum_{[V']} \mathcal X_i^\pm(V, V', W)^{\otimes f^\pm_i(W)+m} \star \delta_* (-1)^{f^\pm_i(W)} \det \mathcal F_i^\pm(V', W)^\vee\, , \quad
\eea
where $\delta : \M(V, W) \hookrightarrow \M(V, W) \times \M(V, W)$ denotes the diagonal embedding and
\be \mathcal X_i^\pm(V, V', W):=\pm q^{-1}\left (\mathcal O_{\M(V, W)} \boxtimes \V' - \V \boxtimes \mathcal O_{\M(V', W)}  \right )|_{\mathfrak C_i^\pm(V', W)} \, .\ee
\end{thm}
\begin{rem}
$\Phi_W$ above is not the homomorphism constructed in \cite{Nakajima}. It is a modified homomorphism due to Varagnolo and Vasserot, \cite{VV}. Note however that both homomorphisms agree on the subalgebra $\mathrm{U}^0_q({\mathrm L}\g)$.
\end{rem}
Letting $K^{G(W)\times \mathbb C^*}(\mathfrak{Z}(W))/\mathrm{torsion}$ be the image of $K^{G(W)\times \mathbb C^*}(\mathfrak{Z}(W)) \to K^{G(W)\times \mathbb C^*}(\mathfrak{Z}(W))\otimes_{\mathbb Z[q,q^{-1}]} \mathbb C(q)$, the following integral restriction of theorem \ref{thm:Nakajima} holds.
\begin{thm}
\label{thm:restNakajima}
The homomorphism $\Phi_{W}$ restricts to a homomorphism $\uqZlg \to K^{G(W)\times \mathbb C^*}(\mathfrak{Z}(W))/\mathrm{torsion}$.
\end{thm}

\section{Graded quiver varieties}
\label{sec:gradedquiver}
When studying standard modules in the next section, we shall use fixed-point subvarieties of $\M(V, W)$ under the action of some abelian reductive subgroup of $G(W)\times \mathbb C^*$. In this section, we thus study those fixed-point subvarieties, leading naturally to the notion of \emph{graded} quiver varieties introduced by Nakajima in \cite{NakqtAxiom}.

\subsection{Fixed point subvarieties}
Let $A$ be an abelian reductive subgroup of $G(W)\times \mathbb C^*$ and denote by $\M(V, W)^A$ and $\M_0(V, W)^A$ the respective $A$-fixed-point subvarieties of $\M(V, W)$ and $\M_0(V, W)$. Let $x \in \M(V, W)^A$ and let $(B, i, j) \in \mu^{-1}(0)^\chi$ be a representative of $x$. For every $a \in A$, there exists $\rho(a) \in G(V)$ such that
\be a * (B, i, j) = \rho(a)^{-1} \cdot (B, i, j) \, . \label{eq:rho}\ee
As $G(V)$ acts freely on $\mu^{-1}(0)^\chi$, this $\rho(a)$ is indeed unique and the map $a \mapsto \rho(a)$ thus defines a group homomorphism $A \rightarrow G(V)$.
\begin{defn}
With every $\rho \in \Hom (A, G(V))$, we associate $\M(V, W)[\rho]$, defined as the set of fixed points in $\M(V, W)^A$ which admit representatives $(B,i,j) \in \mu^{-1}(0)^\chi$ such that (\ref{eq:rho}) above holds.
\end{defn}
Note that, actually, $\M(V, W)[\rho]$ only depends on the conjugacy class $[\rho]$ of $\rho$ in $G(V)$, which is also why we adopt this notation. For every $\rho \in \Hom(A, G(V))$, we further define
\be \LL(V, W)[\rho]:=  \M(V, W)[\rho] \cap \LL(V, W)\, .\ee
\begin{prop}
\label{prop:conndecM}
For every pair of $I$-graded $\C$-vector spaces $(V, W)$,
\be \M(V,W)^A = \bigsqcup_{\rho \in \Hom(A, G(V))} \M(V,W)[\rho] \qquad \mbox{and} \qquad \LL(V,W)^A = \bigsqcup_{\rho \in \Hom(A, G(V))} \LL(V,W)[\rho] \ee
constitute decompositions into connected components.
\end{prop}
\begin{proof}
If $(I,E)$ is of type $\mathfrak a$, $\mathfrak d$, or $\mathfrak e$ and $1 \notin \varepsilon^{\mathbb Z}$, $\M(V, W)[\rho]$ is either empty or connected -- see theorem 5.5.6 in \cite{Nakajima} -- and $\M(V, W)[\rho]$ is homotopic to $\LL(V, W)[\rho]$ -- see proposition 4.1.2 in \cite{Nakajima}.
\end{proof}

\subsection{$A$-weight decompositions}
Regarding $V$ as an $A$-module through $\rho \in \Hom (A, G(V))$, we have the following weight decomposition
\be\label{eq:Vdecomp} V= \bigoplus_{\lambda \in \Sp(\rho, V)} V(\lambda) \, ,\ee
where $\lambda \in \Sp(\rho, V) \subset \Hom (A, \mathbb C^*)$ if and only if 
\be V(\lambda) := \{v \in V:\forall a \in A \quad  \rho(a) v = \lambda(a) v \}\ee
is non-zero. Similarly, $W$ can be regarded as an $A$-module through $A \hookrightarrow G(W) \times \mathbb C^* \twoheadrightarrow G(W)$ and we denote by $W(\lambda)$ the corresponding weight spaces for $\lambda \in \Sp(W)$.
We let $q \in \Hom (A, \C^*)$ be the composite $q: A \hookrightarrow G(W) \times \C^* \twoheadrightarrow \C^*$, where the second arrow is projection on the second factor. Then, we have
\begin{prop}
\label{VWspec}
For every $\rho \in \Hom(A, G(V))$, 
\be \label{eq:specaVaW} \Sp(\rho, V) \subseteq q^\Z \Sp (W) \, .\ee
\end{prop}
\begin{proof}
By definition, each point in $\M(V, W)[\rho]$ admits a representative $(B, i, j) \in \mu^{-1}(0)^\chi$ such that (\ref{eq:rho}) holds or, equivalently, such that
\be B_h(V_{\out(h)}(\lambda)) \subseteq V_{\inn(h)}(q^{-1}\lambda) \, , \qquad i_k(W_k(\lambda)) \subseteq V_k(q^{-1} \lambda) \,, \qquad j_k(V_k(\lambda)) \subseteq W_k(q^{-1} \lambda)\,, \label{eq:inclusions}\ee
for all $h \in H$, $k \in I$ and $\lambda \in \Hom(A, \C^*)$. Let $S:=\bigoplus_{k \in I} S_k$ with
\be S_k := \bigoplus_{\substack{\lambda \in \Hom(A, \C^*)\\ W(q^{\Z} \lambda)=0}} V_k(\lambda) \,,\ee
for every $k \in I$. Clearly, $S \subseteq \ker j$ is $B$-invariant and thus $S=0$ by proposition \ref{prop:stab}.
\end{proof}
For every $\rho \in \Hom (A, G(V))$, denote by $i_\rho : \M(V, W)[\rho] \hookrightarrow \M(V, W)$ the inclusion. Clearly, the restrictions $i^*_\rho \V$ and $i^*_\rho\W$ are bundles of $A$-modules through $\rho \times \id_{G(W) \times \C^*} : A \rightarrow G(V) \times G(W) \times \C^*$ and they admit similar weight decompositions
\be \label{TautDecomp} i^*_\rho \V= \bigoplus_{\lambda \in \Sp(\rho, V)} i^*_\rho \V(\lambda) \qquad \qquad i^*_\rho \W =  \bigoplus_{\lambda \in \Sp(W)} i^*_\rho \W(\lambda) \, .\ee
Eventually, we have, for all $k \in I$,
\be\label{FDecomp} i^*_\rho \F_k(V, W) = \bigoplus_{\lambda \in q^\Z\Sp(W)} i^*_\rho \F_k(V, W)(\lambda)\,.\ee

\subsection{Graded quiver varieties}
We now introduce \emph{graded} quiver varieties as a natural parametrization of the $A$-fixed point subvarieties $\M(V, W)[\rho]$. Given $V$, $W$ and $\rho \in \Hom(A, G(V))$ as in the previous section, we now regard $V$ and $W$ as finite dimensional $I \times \Hom(A, \C^*)$-graded vector spaces,
\be \label{VWDecomp} V = \bigoplus_{(k,\lambda) \in I \times \Hom(A, \C^*)} V_k(\lambda)\,, \qquad \qquad W = \bigoplus_{(k,\lambda) \in I \times \Hom(A, \C^*)} W_k(\lambda) \,.\ee
Define the affine algebraic variety $M^\bullet(V, W)$ as the subset of $M(V, W)$ consisting of those triples $(B,i,j)$ such that (\ref{eq:inclusions}) above holds and denote by 
\bea &B_{h, \lambda} : V_{\out(h)}(\lambda) \rightarrow  V_{\inn(h)}(q^{-1} \lambda)& \nn\\
&i_{k,\lambda} : W_k(\lambda) \rightarrow V_k(q^{-1}\lambda) & \nn\\
&j_{k, \lambda} : V_k(\lambda) \rightarrow W_k(q^{-1}\lambda) &
\eea
the respective restrictions of $B_h$, $i_k$ and $j_k$ to the corresponding components of $V$ and $W$ -- see decomposition (\ref{eq:Vdecomp}) and eq. (\ref{eq:inclusions}). Let similarly
\be \mu_{k, \lambda} := \sum_{\substack{h \in H\\ \inn(h)=k}} \varepsilon(h) B_{h, q^{-1}\lambda} \circ B_{\bar{h}, \lambda} + i_{k, q^{-1}\lambda} \circ j_{k , \lambda}\ee
be the $(k, \lambda)$-component of the restriction $\mu_\bullet := \mu|_{M^\bullet(V, W)}$ of the momentum map $\mu$ defined in (\ref{eq:defmu}). Consider the affine algebraic variety $\mu^{-1}_\bullet(0) \subseteq M^\bullet(V, W)$. It is acted upon by $G^\bullet(V) := \prod_{(k,\lambda) \in I \times \Hom(A, \C^*)} GL(V_k(\lambda))$, \ie the maximal subgroup of $G(V)$ preserving the $I \times \Hom(A, \C^*)$-grading of $V$. The character $\chi$ defined in section \ref{quivers} clearly restricts to $G^\bullet(V)$, where it factors according to
\be\label{charfact}\chi(g) = \prod_{(k,\lambda) \in I \times \Hom(A, \C^*)} \det g_{k, \lambda}^{-1}\, .\ee
for all $g \in G^\bullet(V)$. It is worth emphasizing that the products in the definition of $G^\bullet(V)$ above and in (\ref{charfact}) only ever contain a finite number of factors as all but a finite number of $V_k(\lambda)$ are non-zero. Applying again the constructions of section \ref{GIT}, we get the following
\begin{defn}
For every pair $V, W$ of finite dimensional $I \times \Hom(A, \C^*)$-graded $\C$-vector spaces, we let
\be \M_0^\bullet(V, W) := \mu^{-1}_\bullet(0) /\!/ G^\bullet(V) \ee
and
\be \M^\bullet(V, W):= \mu^{-1}_\bullet(0) /\!/_\chi \, G^\bullet(V) \, .\ee
The latter is referred to as a \emph{graded quiver variety} -- see \eg \cite{NakqtAxiom}.
\end{defn}
By the results of section \ref{GIT}, there exists a projective morphism $\pi_\bullet: \M^\bullet(V, W) \to \M^\bullet_0(V, W)$, which coincides with $\pi|_{\M(V, W)[\rho]}$, the restriction to $\M(V, W)[\rho]$ of the projective morphism of proposition \ref{Prop:projmorph}. It follows that if we let
\be \LL^\bullet(V, W)_x:= \pi_\bullet^{-1}(x)\ee
denote its fibre at $x \in \M^\bullet_0(V, W)$, we get $\LL_\bullet(V, W)_x \cong \LL(V, W)[\rho]$. Note that, by construction, $\LL^\bullet(V, W)_x$ is projective. It follows, either by restriction or by use of the results of section \ref{GIT}, that the conclusions of proposition \ref{prop:canampleline} hold for graded quiver varieties. In particular, we have
\begin{prop}
\label{prop:gradquivquot}
The following hold:
\begin{enumerate}
\item[i.] $\mu^{-1}_\bullet(0)^\chi$ is a principal $G^\bullet(V)$-bundle over $\mu^{-1}_\bullet(0)^\chi/G^\bullet(V)$;
\item[ii.] $\M^\bullet(V, W)$ is isomorphic to the non-singular quasi-projective variety $\mu^{-1}_\bullet(0)^\chi/G^\bullet(V)$.
\end{enumerate}
\end{prop}
\begin{rem}
Proposition \ref{prop:stab} provides a criterion for the stability of any $(B,i,j) \in \mu^{-1}(0)$ which can be extended verbatim to $M(V, W)$ and then restricted to $M^\bullet(V, W)$. It is clear that the stability condition thus obtained is equivalent to the one referred to in the above proposition \ref{prop:gradquivquot}.
\end{rem}
The relevance of graded quiver varieties to the present situation stems from the following
\begin{prop}
For every pair $V, W$ of finite dimensional $I$-graded $\C$-vector spaces and every $\rho \in \Hom(A, G(V))$, we have
\be \M(V, W)[\rho] \cong \M^\bullet(V, W)\, ,\ee
where, on the right hand side, $V$ and $W$ are regarded as $I \times \Hom(A, \C^*)$-graded vector spaces, according to the decompositions (\ref{VWDecomp}).
\end{prop}
\begin{proof}
It is obvious from the definitions that
\be \M(V, W)[\rho] \cong \mu^{-1}_\bullet(0)^\chi/G^\bullet(V)\ee
and the result follows from proposition \ref{prop:gradquivquot}.
\end{proof}

\subsection{Ample line bundles} 
Given $V$, $W$ and $\rho \in \Hom(A, G(V))$ as in the previous section, we associate with every triple $(B,i,j) \in M^\bullet(V, W)$ the following map
\be\sigma_{k, \lambda} := \bigoplus_{\substack{h \in H \\ \out(h)=k}} B_{h,\lambda} \oplus j_{k, \lambda} : V_k(\lambda) \rightarrow \bigoplus_{\substack{h \in H \\ \out(h)=k}} V_{\inn(h)}(q^{-1} \lambda) \oplus W_k(q^{-1}\lambda)\, .\ee
Let furthermore $\mu^{-1}_{k, \lambda}(0) \subseteq M^\bullet(V, W)$ denote the zero set of $\mu_{k, \lambda}$.
\begin{defn}
For every $k \in I$ and every $\lambda \in \Hom(A, \C^*)$, define the following varieties:
\bea \Mtilde(k, \lambda)&:=& \{(B, i, j) \in \mu_{k,\lambda}^{-1}(0) : \ker \sigma_{k, \lambda} =0 \} / GL(V_k(\lambda))\,, \\
\Mtildeo(k, \lambda)&:=& \mu_{k,\lambda}^{-1}(0)^\chi / GL(V_k(\lambda))\,,\\
\Mhat^\bullet(k, \lambda)&:=& \mu^{-1}_\bullet(0)^\chi/GL(V_k(\lambda))
\eea
and
\bea
M'(k,\lambda)&:=& \bigoplus_{\substack{(h, \lambda') \in H \times \Hom(A, \C^*)\\ (k, \lambda) \notin \{(\inn(h), \lambda'), (\out(h), \lambda')\}}} \Hom(V_{\out(h)}(\lambda'), V_{\inn(h)}(q^{-1}\lambda')) \nn \\
& & \oplus \bigoplus_{(k', \lambda')\neq (k,q^{-1}\lambda)} \Hom(W_{k'}(\lambda'), V_{k'}(q^{-1}\lambda') ) \oplus \bigoplus_{(k',\lambda') \neq(k,\lambda)} \Hom(V_{k'}(\lambda'), W_{k'}(q^{-1}\lambda'))\,.\qquad 
\eea
\end{defn}

\begin{prop}
\label{prop:principalGklbundle}
For every $k \in I$ and every $\lambda \in \Hom(A, \C^*)$, we have the following diagram
$$
\begin{CD}
\Mhat^\bullet(k, \lambda)  @>i>> \Mtildeo (k, \lambda) @>j>> \Mtilde(k,\lambda) @= \M(V_k(\lambda), E_{k, \lambda}) \times  M'(k,\lambda) \\
@V\pi_{k,\lambda}VV &&&& @VVp_1V\\
\M(V, W)[\rho] &&&&&& \M(V_k(\lambda), E_{k, \lambda})
\end{CD} 
$$
where $i$ (resp. $j$) is a closed (resp. open) immersion, $p_1$ denotes projection on the first factor and
\be E_{k,\lambda}:= \bigoplus_{\substack{h \in H\\ \out(h)=k}} V_{\inn(h)}(q^{-1}\lambda) \oplus W_k(q^{-1}\lambda) \, .\ee
The left vertical arrow makes $\Mhat(k, \lambda)$ a principal $G_{k,\lambda}:=\prod_{(k',\lambda') \neq (k,\lambda)} GL(V_{k'}(\lambda'))$-bundle over $\M(V, W)[\rho]$.
\end{prop}
\begin{proof}
$\Mtildeo(k,\lambda)$ clearly is an open subvariety of $\Mtilde(k, \lambda)$. $\Mhat(k, \lambda)$ is, in turn, the closed algebraic subvariety of $\Mtildeo(k,\lambda)$ defined by the equations $\mu_{k', \lambda'}=0$ for all $(k', \lambda') \neq (k, \lambda)$. Finally, $\M(V, W)[\rho] = \Mhat(k,\lambda)/G_{k,\lambda}$ by proposition \ref{prop:gradquivquot}.
\end{proof}

\begin{rem}
For every pair of integers $m,n \in \N$, denote by $\mathrm{Gr}(m,n)$ the Grassmannian manifold of $m$-dimensional subspaces of $\C^n$ and let $T^*\mathrm{Gr}(m,n)$ denote its cotangent bundle. Then, it is clear from the definitions that $\M(V_k( \lambda), E_{k,\lambda} ) \cong T^*\mathrm{Gr}(\dim V_k(\lambda),\dim E_{k,\lambda})$ and that we recover the graded Grassmannians relevant in type $\mathfrak a_1$ -- see \cite{YZ}.
\end{rem}

\begin{defn}
For every $k \in I$ and every $\lambda \in \Hom(A, \C^*)$, we define the following vector bundles:
\bea \Vtilde_k(\lambda) &:=& \mu_{k,\lambda}^{-1}(0) \times_{GL(V_k(\lambda))} V_k(\lambda) \rightarrow \Mtilde(k, \lambda)\,, \\
\Vtildeo_k(\lambda) &:=& \mu_{k,\lambda}^{-1}(0)^\chi \times_{GL(V_k(\lambda))} V_k(\lambda) \rightarrow \Mtildeo(k, \lambda) \,,\\
\Vhat_k(\lambda)&:=& \mu^{-1}_\bullet(0)^\chi \times_{GL(V_k(\lambda))}  V_k(\lambda) \rightarrow \Mhat^\bullet(k, \lambda)\,.
\eea
\end{defn}

We now establish the most important result of this section.

\begin{lem}
\label{lem:Vample}
We have
\begin{enumerate}
\item[i.] $\det \Vtilde_k(\lambda)^\vee \cong  p_1^* \mathcal O_{\M(V_k( \lambda), E_{k,\lambda} )}(1)$;
\item[ii.] $\det \Vhat_k(\lambda)^\vee \cong i^* \circ j^* \det \Vtilde_k(\lambda)^\vee$;
\item[iii.] $\det \Vhat_k(\lambda)^\vee \cong \pi_{k,\lambda}^* \det i_\rho^*\V_k(\lambda)^\vee$;
\item[iv.] $\det i_\rho^* \V_k(\lambda)^\vee \in \mathrm{Pic}(\M(V, W)[\rho])$ is an ample invertible sheaf.
\end{enumerate}
\end{lem}
\begin{proof}
\emph{i.}, \emph{ii.} and \emph{iii.} are obvious. We prove \emph{iv.} For every $s \in \bigoplus_{n \in \N^*} H^0(\M(V_k( \lambda), E_{k,\lambda} ), \mathcal O_{\M(V_k( \lambda), E_{k,\lambda} )}(1)^{\otimes n})$, let
\be \M(V_k( \lambda), E_{k,\lambda} )_s:= \{ x \in \M(V_k( \lambda), E_{k,\lambda} ) : s(x) \neq 0\}\, .\ee
It follows from the ampleness of $\mathcal O_{\M(V_k( \lambda), E_{k,\lambda} )}(1)$, proposition \ref{prop:Mquot}, that those of the $\M(V_k( \lambda), E_{k,\lambda} )_s$ that are affine constitute an open affine covering of $\M(V_k( \lambda), E_{k,\lambda} )$ -- see \eg EGA II, Theorem 4.5.2.a'.
Obviously, the $p_1^{-1}(\M(V_k( \lambda), E_{k,\lambda} )_s) = \M(V_k( \lambda), E_{k,\lambda} )_s \times M'(k,\lambda)$ constitute an open affine covering of $\M(V_k(\lambda), E_{k, \lambda}) \times  M'(k,\lambda)$. By virtue of \emph{i.}, every $s \in \bigoplus_{n \in \N^*} H^0(\M(V_k( \lambda), E_{k,\lambda} ), \mathcal O_{\M(V_k( \lambda), E_{k,\lambda} )}(1)^{\otimes n})$ pulls back to a section $p_1^*s \in \bigoplus_{n \in \N^*} H^0(\M(V_k( \lambda), E_{k,\lambda} ), (\det \Vtilde_k(\lambda)^\vee)^{\otimes n})$ in such a way that
\be (\M(V_k(\lambda), E_{k, \lambda}) \times  M'(k,\lambda))_{p_1^*s} :=  \{x \in \M(V_k(\lambda), E_{k, \lambda}) \times  M'(k,\lambda): p_1^*s (x) \neq 0\} = \M(V_k( \lambda), E_{k,\lambda} )_s \times M'(k,\lambda)\, .\nn\ee
Thus, those $(\M(V_k(\lambda), E_{k, \lambda}) \times  M'(k,\lambda))_{p_1^*s}$ that are affine cover $\M(V_k(\lambda), E_{k, \lambda}) \times  M'(k,\lambda)$ and $\det \Vtilde_k(\lambda)^\vee$ is ample. Restricting $\det \Vtilde_k(\lambda)^\vee$ to the locally closed subvariety $\Mhat^\bullet(k, \lambda)$ yields an ample invertible sheaf and, by \emph{ii.}, it follows that $\det \Vhat_k(\lambda)^\vee$ is ample.

Now, let $U \subseteq \M(V, W)[\rho]$ be any Zariski-open set. By virtue of proposition \ref{prop:principalGklbundle}, $\Mhat^\bullet(k,\lambda)$ is a principal $G_{k,\lambda}$-bundle over $\M(V, W)[\rho]$ with projection $\pi_{k,\lambda} : \Mhat^\bullet(k,\lambda) \rightarrow \M(V, W)[\rho]$ and it follows that $\pi_{k, \lambda}^{-1}(U) \subseteq \Mhat^\bullet(k, \lambda)$ is a Zariski-open set. Since $\det \Vhat_k(\lambda)^\vee$ is ample, there exists a non-empty $S \subseteq \bigoplus_{n \in \N^*} H^0(\Mhat^\bullet(k, \lambda), (\det \Vhat_k(\lambda)^\vee)^{\otimes n})$ such that
\be \pi_{k, \lambda}^{-1}(U) = \bigcup_{s \in S} \Mhat^\bullet(k,\lambda)_s\, ,\ee
where we have set
\be \Mhat^\bullet(k,\lambda)_s := \{x \in \Mhat^\bullet(k,\lambda) : s(x) \neq 0\}\, .\ee
Observe that the sections in $\bigoplus_{n \in \N^*} H^0(\Mhat^\bullet(k, \lambda), (\det \Vhat_k(\lambda)^\vee)^{\otimes n})$ are constant on $G_{k,\lambda}$-orbits. Hence, each section $s \in S$ defines a unique section $t \in \bigoplus_{n \in \N^*} H^0(\M(V, W)[\rho], (\det \V_k(\lambda)^\vee)^{\otimes n})$ by setting
\be t(G_{k,\lambda} . x) := s(x)\, , \ee
for every $G_{k, \lambda}$-orbit in $\M(V, W)[\rho]$. Then, it is clear that
\be \M(V, W)[\rho]_t := \{x \in \M(V, W)[\rho] : t(x) \neq 0\} = \pi_{k,\lambda}( \Mhat^\bullet(k, \lambda)_s)\ee
and that
\be U = \bigcup_{s \in S} \pi_{k, \lambda}(\Mhat^\bullet(k,\lambda)_s) = \bigcup_{t: \pi_{k,\lambda}^*t \in S} \M(V, W)[\rho]_t\, .\ee
We have thus proven that, as $t$ runs through $\bigoplus_{n \in \N^*} H^0(\M(V, W)[\rho], (\det \V_k(\lambda)^\vee)^{\otimes n})$, the $\M(V, W)[\rho]_t$ constitute a basis for the Zariski topology of $\M(V, W)[\rho]$ and hence that $\det \V_k(\lambda)^\vee$ is ample -- see \eg \nolinebreak{EGA II}, Theorem 4.5.2.a.
\end{proof}

\subsection{Factorization into $q$-segments} \label{sec:qseg}
We define an equivalence relation $\sim$ on $\Hom(A, \C^*)$ by setting $\lambda_1 \sim \lambda_2$ if and only if $\lambda_1 \in q^\Z \lambda_2$ and we let 
\be \Sp(W) = \bigsqcup_{\Sigma \in \Sp(W)/\sim} \Sigma \label{qseg}\ee
be the decomposition of $\Sp(W)$ into equivalence classes. Elements of $\Sp(W)/\sim$ will be referred to as \emph{$q$-segments} of $\Sp(W)$. The \emph{length} of a given $q$-segment is, by definition, the number of distinct elements in the corresponding equivalence class. A similar decomposition into $q$-segments obviously holds for $\Sp(\rho, V)$ and, by virtue of proposition \ref{VWspec}, we have an injective map $\varsigma : \Sp(\rho, V)/\sim \,\, \hookrightarrow \,\, \Sp(W)/\sim$. Each $q$-segment $\Sigma$ of either $\Sp(W)/\sim$ or $\Sp(\rho, V)/\sim$ is finite and -- provided there exists $(\id_W, \varepsilon) \in A$ such that $1 \notin \varepsilon^\Z$ -- totally ordered by setting $\lambda_1 < \lambda_2$ for all $\lambda_1, \lambda_2 \in \Sigma$ such that $\lambda_1 \in q^{\N^*} \lambda_2$. 

 The decomposition (\ref{qseg}) induces the decomposition
\be W = \bigoplus_{\Sigma \in \Sp(W)/\sim} W_\Sigma \qquad \mbox{where} \qquad W_\Sigma := \bigoplus_{\lambda \in \Sigma} W(\lambda) \ee
and, subsequently, $A= \prod_{\Sigma \in \Sp(W)/\sim} A_\Sigma$ where $A_\Sigma \subset G(W_\Sigma) \times \C^*$ for all $\Sigma \in \Sp(W)/\sim$. Similarly, denoting by $[\lambda] \in \Sp(\rho, V)/ \sim$ the equivalence class of $\lambda \in \Sp(\rho, V)$, we set, for all $\Sigma \in \Sp(W)/\sim$,
\be V_\Sigma := \bigoplus_{\substack{\lambda \in \Sp(\rho, V)\\ \varsigma([\lambda]) = \Sigma}} V(\lambda)\, , \qquad \mbox{so that} \qquad V= \bigoplus_{\Sigma \in \Sp(W)/\sim} V_{\Sigma}\, .\ee
For every $\Sigma \in \Sp(W)/\sim$, $V_{\Sigma}$ is an $A_\Sigma$-module through $\rho_\Sigma \in \Hom(A_\Sigma, G(V_{\Sigma}))$ and $\rho= \bigoplus_{\Sigma \in \Sp(W)/\sim} \rho_\Sigma$.

\begin{prop} \label{prop:qsegdecomp} Let $W=\bigoplus_\Sigma W_\Sigma$ and $V= \bigoplus V_\Sigma$ be the respective decompositions of $W$ and $V$ into $q$-segments as above. We have
\be \label{eq:MaMSigma} \M(V,W)^A \cong \prod_{\Sigma \in \Sp(W)/\sim} \M(V_\Sigma, W_\Sigma)^{A_\Sigma}\,.\ee
Furthermore, for every $\rho = \bigoplus_\Sigma \rho_\Sigma \in \Hom (\prod_\Sigma A_\Sigma, \prod_\Sigma G(V_\Sigma))$, we even have
\be \label{eq:MaMSigmarho} \M(V,W)[\rho] \cong \prod_{\Sigma \in \Sp(W)/\sim} \M(V_\Sigma, W_\Sigma)[\rho_\Sigma]\,.\ee
\end{prop}
\begin{proof}
Left to the reader.
\end{proof}

For every $\Sigma \in \Sp(W) / \sim$ let
\be \V_\Sigma := \mu_\Sigma^{-1}(0)^s \times_{G(V_\Sigma)} V_\Sigma\ee
and denote by $i_{\rho_\Sigma}: \M(V_\Sigma, W_\Sigma)[\rho_\Sigma] \hookrightarrow \M(V_\Sigma, W_\Sigma)$ the natural inclusion.

\begin{prop} \label{Prop:Comult}
For every $\rho = \bigoplus_\Sigma \rho_\Sigma \in \Hom (\prod_\Sigma A_\Sigma, \prod_\Sigma G(V_\Sigma))$, we have
\be i_\rho^* \V = \bigoplus_{\Sigma \in \Sp(W)/\sim} i_{\rho_\Sigma}^*\V_\Sigma \boxtimes \mathcal O_{\prod_{\Sigma' \neq \Sigma} \M(V_{\Sigma'}, W_{\Sigma'})[\rho_{\Sigma'}] } \ee
\end{prop}
\begin{proof} Left to the reader.
\end{proof}

\subsection{Bia\l ynicki-Birula type decomposition for individual $q$-segments}
In general, within any given $q$-segment of length strictly greater than $1$, further factorizations like (\ref{eq:MaMSigma}) and (\ref{eq:MaMSigmarho}) do not hold. A $\C^*$-action can nonetheless be defined on the relevant quiver varieties, that admits as its fixed point set the desired factorization, and that allows one to perform a decomposition \`a la Bia\l ynicki-Birula with respect to that fixed point set. Assume indeed that the abelian reductive subgroup $A$ of $G(W)\times\C^*$ is such that $\Sp(W)=\Sigma$ consists of a single $q$-segment of length $l \in \N^*$. Write $\Sigma=\{\lambda_1, \dots, \lambda_l\}$, the numbering being chosen in such a way that $\lambda_i <\lambda_{i+1}$ for all $i\in \range{1}{l}$. Let for short $W_i:=W(\lambda_i)$, so that the decomposition of $W$ into eigenspaces reads $W=W_1 \oplus W_2 \oplus \cdots \oplus W_l$. Write similarly $A_i:=A|_{W_i}=\lambda_i \,\id_{W_i} \in G(W_i)\times \Cx$ for every $i\in\range{1}{l}$, so that $A=(A_1, \dots , A_l)$. Define then a group homomorphism $\gamma : \C^* \rightarrow G(W)\times \Cx$ by setting
\be \gamma(t) = t^{m_1} \id_{W_1} \oplus \cdots \oplus t^{m_l} \id_{W_l} \,,\ee
for some fixed choice of $m_1<m_2<\cdots <m_l$. It is easily shown that for every $I$-graded $\C$-vector space $V$
\bea \bigsqcup_{\substack{[V_1],\dots , [V_l]\\V\cong V_1\oplus \cdots \oplus V_l}} \M(V_1, W_1)^{A_1} \times  \cdots \times \M(V_l, W_l)^{A_l} &\rightarrow &\M(V, W)^A\nn\\
\left ([B_1, i_1, j_1], \dots, [B_l, i_l, j_l]\right ) &\mapsto & [B_1\oplus \cdots \oplus  B_l, i_1\oplus \cdots \oplus i_l, j_1 \oplus \cdots\oplus j_l]\nn\eea
defines a closed embedding whose image is $\M(V, W)^{\left \langle A, \gamma(\C^*)\right \rangle}$, where $\left \langle A, \gamma(\C^*)\right \rangle$ denotes the Zariski closed subgroup of $G(W)\times\C^*$ generated by $A$ and $\gamma(\C^*)$. 
We can thus identify the fixed point set $\M(V, W)^{\left \langle A, \gamma(\C^*)\right \rangle}$ with $\M(V_1, W_1)^{A_1} \times \cdots \times \M(V_l, W_l)^{A_l}$ and similarly $\LL(V,W)^{\left \langle A,\gamma(\Cx)\right \rangle}$ with $\LL(V_1, W_1)^{A_1} \times \cdots \times \LL(V_l, W_l)^{A_l}$. The latter can further be decomposed into connected components as in proposition \ref{prop:conndecM}, \ie
\be \LL(V_1, W_1)^{A_1} \times \cdots \times  \LL(V_l, W_l)^{A_l}= \bigsqcup_{\substack{\rho_1\in\Hom(A_1, G(V_1))\\ \vdots \\ \rho_l\in\Hom(A_l, G(V_l))}} \LL(V_1, W_1)[\rho_1] \times \cdots \times \LL(V_l, W_l)[\rho_l]\,.\ee
One can therefore define the attracting sets
\be \LL^+(V_1, W_1; \cdots ;V_l, W_l)[\rho_1,\dots, \rho_l] := \left \{x\in\M(V, W)^A: \lim_{t\to 0} \gamma(t) * x \in \prod_{i=1}^l\LL(V_i, W_i)[\rho_i] \right \}\,.\ee
\begin{prop}
\label{prop:BBdecomp}
With $A$ and $W=W_1 \oplus \cdots \oplus W_l$ as above, let $V$ be some $I$-graded $\C$-vector space and let $\rho \in \Hom(A,G(V))$. Then,
\begin{enumerate}
\item[i.] for every $l$-tuple of $I$-graded $\C$-vector spaces $(V_1, \dots , V_l)$ such that $V\cong V_1 \oplus \cdots \oplus V_l$ and every $l$-tuple $(\rho_1, \dots, \rho_l) \in \Hom(A, G(V_1)) \times \cdots \times \Hom(A, G(V_l))$ such that $\rho = \rho_1 \oplus \cdots \oplus \rho_l$, $\LL^+(V_1, W_1; \cdots ;V_l, W_l)[\rho_1,\dots, \rho_l]$ is a nonsingular locally closed subvariety of $\M(V,W)[\rho]$, isomorphic to the total space of the subbundle $T\M(V,W)[\rho]|_{\LL(V_1, W_1)^{A_1} \times \cdots \times \LL(V_l, W_l)^{A_l}}^+$ of the restricted tangent bundle $T\M(V,W)[\rho]|_{\LL(V_1, W_1)^{A_1} \times \cdots \times \LL(V_l, W_l)^{A_l}}$, whose sections transform with strictly positive weights under $\gamma(\Cx)$;
\item[ii.] the decomposition
\be \LL(V,W)[\rho] = \bigsqcup_{\substack{[V_1],\dots , [V_l]\\V\cong V_1\oplus \cdots \oplus V_l}} \bigsqcup_{\substack{\rho_1\in\Hom(A_1, G(V_1))\\ \vdots \\ \rho_l\in\Hom(A_l, G(V_l))}} \LL^+(V_1, W_1; \cdots ;V_l, W_l)[\rho_1,\dots, \rho_l]\label{Eq:LBBdecomp}\ee
is \emph{good} in the terminology of \cite{CarrellGoresky}.
\end{enumerate}
\end{prop}

\begin{proof}
A proof of \emph{i.} is obtained by restricting results in \cite{NakTens} to the the appropriate connected components of the fixed point set. Similarly, eq. (\ref{Eq:LBBdecomp}) can be proven by adapting an argument in \cite{NakTens} and \emph{i.} then implies that this decomposition is good -- see \eg \cite{BBCMG}. Thus \emph{ii.} follows.
\end{proof}

\section{Standard modules over simply laced quantum affine algebras}
\label{sec:stdmodules}
Before turning to the definition of standard modules, let us observe that
\be R(G(W)\times \C^*) \ni x  \mapsto x \otimes \sum_{[V]} \mathcal O_{\delta(\M(V, W))}\ee
defines an injective homomorphism from $R(G(W)\times \C^*)$ to the center of $K^{G(W) \times \C^*}(\ZZ(W))$, regarded as an algebra by virtue of the construction of section \ref{Sec:Convolution}. Therefore, $R(G(W)\times \C^*)$ acts on any finite dimensional simple $K^{G(W) \times \C^*}(\ZZ(W))$-module by a homomorphism $\chi :R(G(W)\times \C^*) \to \C$. Such a homomorphism $\chi$ is given by evaluation of the character at a semisimple element $a=(s, \varepsilon) \in G(W) \times \C^*$.

\subsection{Definition}Throughout this section $a = (s, \varepsilon)$ denotes a semisimple element of $G(W) \times \mathbb C^*$ and $A$ the Zariski closure of $a^{\mathbb Z}$. $A$ is an abelian reductive subgroup of $G(W) \times \mathbb C^*$. Let $R(A)$ denote the representation ring of $A$, \ie the Grothendieck ring of the category of finite dimensional rational representations of $A$. Given $a$, we let $\chi_a : R(A) \rightarrow \mathbb C$ denote the homomorphism defined by evaluating characters at $a$ and $R(A)_a$, the localization of $R(A)$ with respect to $\ker \chi_a$. We shall denote by $\mathbb C_a$ the one-dimensional $R(A)$-module, constructed on $\mathbb C$ \emph{via} $\chi_a$. Note that $\mathbb C_a$ can also be regarded as a $G(W) \times \mathbb C^*$-module \emph{via} the homomorphism $R(G(W) \times \mathbb C^*)\rightarrow R(A)$.

Let $x \in \M_0^{\mathrm{reg}}(V^0, W)$ be fixed under the action of $A$. Obviously, $\LL(V, W)_x$ is invariant under the action of $A$. We let
\be \LL(W)_x := \bigsqcup_{[V]} \LL(V, W)_x \ee
and set
\be K^A(\LL(W)_x) := \bigoplus_{[V]} K^A(\LL(V, W)_x) \, .\ee
\begin{defprop}
For every $x \in \M_0^{\mathrm{reg}}(V^0, W)$ and $a \in G(W) \times \mathbb C^*$ as above, 
\be {\bf M}_{x, a} := K^A(\LL(W)_x) \otimes_{R(A)} \mathbb C_a\ee
is a finite dimensional $\uepslg$-module which we will refer to as a \emph{standard module}.
\end{defprop}
\begin{proof}
$K^A(\LL(W)_x)$ is a $K^A(\ZZ(W))$-module through the convolution product. That structure descends to a $K^A(\ZZ(W))/\mathrm{torsion} \otimes_{R(A)} \C_a$-module structure on ${\bf M}_{x, a}$ -- see \cite{Nakajima} for details. It follows from theorem \ref{thm:restNakajima} that ${\bf M}_{x, a}$ is a $\uqZlg$-module and, since $q$ acts over ${\bf M}_{x,a}$ by multiplication by $\varepsilon$, that it is, indeed, an $\uepslg$-module as claimed. That it is finite dimensional follows from the fact that there are only finitely many equivalence classes of $I$-graded vector spaces $[V]$ such that $\LL(V, W)$ be non-empty.
\end{proof}
\begin{rem}
It follows from proposition \ref{prop:regular} that every $x \in \M_0(V_0, W)$ is actually regular for $(I, E)$ of type $\mathfrak a$, $\mathfrak d$, or $\mathfrak e$. Moreover, it turns out that the ${\bf M}_{0,a}$ with $a \in G(W) \times \C^*$ semisimple form a basis of the Grothendieck ring $\GrMod{\uepslg}$ of the category of finite dimensional $\uepslg$-modules. From now on, we therefore restrict our attention to those standard modules and denote them simply by ${\bf M}_a$. 
\end{rem}

\subsection{Localization}
Given any $R(A)$-module $M$,  denote by $M_a$ its localization $M \otimes_{R(A)}R(A)_a$ at $a$. Let furthermore  $i: \M(W)^A \rightarrow\M(W)$ be the natural inclusion. It induces inclusions $\M(W)^A \times \M(W)^A \rightarrow \M(W)\times \M(W)$ and hence $\ZZ(W)^A \rightarrow \ZZ(W)$ that we shall also denote by $i$, the meaning of $i$ in each case being clear from the context. It follows from Thomason's localization theorem \cite{Thomason} that for $X= \LL(W)$ or $\ZZ(W)$,
\be i^*: K^A(X)\otimes_{R(A)} R(A)_a \stackrel{\sim}{\longrightarrow} K^A(X^A) \otimes_{R(A)} R(A)_a\ee
is an isomorphism. However, when $X=\ZZ(W)$, $i^*$ is not compatible with convolution. Following \cite{CG}, we thus let $\lambda_A:= {\bigwedge}_{-1}N^*$, where $N$ is the normal bundle of $\M(W)^A$ in $\M(W)$. $\lambda_A$ is invertible in the localized $K$-group $K^A(\M(W)^A)_a$ and we denote by $\lambda_A^{-1}$ its inverse.
\begin{prop} The \emph{bivariant localization map}
\be r_a:= \left ( 1 \boxtimes \lambda_A^{-1} \right ) \circ i^* \ee
is an isomorphism $K^A(\ZZ(W))_a \stackrel{\sim}{\longrightarrow} K^A(\ZZ(W)^A)_a$ and it is compatible with convolution in the sense that the following two diagrams commute:
$$
\begin{CD}
K^A(\ZZ(W))_a ^{\otimes_{R(A)_a} 2} @>\star>> K^A(\ZZ(W))_a\\
@V{r_a \otimes r_a}V{\cong}V @V{\cong}V{ r_a}V\\
K^A(\ZZ(W)^A)_a^{ \otimes_{R(A)_a} 2} @>\star>> K^A(\ZZ(W)^A)_a
\end{CD} 
\qquad\quad  
\begin{CD}
{}\\
K^A(\ZZ(W))_a \underset{R(A)_a}{\otimes}
K^A(\LL(W))_a @>\star>> K^A(\LL(W))_a\\
@V{r_a \otimes i^*}V{\cong}V @V{\cong}V{i^*}V\\
K^A(\ZZ(W)^A)_a \tens{R(A)_a} K^A(\LL(W)^A)_a @>\star>> K^A(\LL(W)^A)_a
\end{CD}
$$
\end{prop}
In both cases, $X= \LL(W)$ or $\ZZ(W)$, the action of $A$ on $X^A$ being trivial, we have $K^A(X^A) \cong K(X^A) \otimes R(A)$. We thus define the evaluation map $\ev_a := \id_{K(X^A)} \otimes \chi_a : K(X^A) \otimes R(A)_a \rightarrow K(X^A) \otimes \C =:K(X^A)_\C$. Explicitly, let $\F \in K^A(X^A)$. $A$ acts on $\F$ fiberwise and $\F$ admits a decomposition $\F = \bigoplus_{\lambda \in \Sp(\F)} \F(\lambda)$ into $A$-weight spaces, from which it follows that
\be \ev_a(\F) = \sum_{\lambda \in \Sp(\F)} \F(\lambda) \otimes \lambda(a) \in K(X^A) \otimes \C\, ,\ee
disregarding the $A$-action on each left tensor factor.

Let $\ch_\bullet : K(\ZZ(W)^A) \rightarrow H_\bullet(\ZZ(W)^A, \Q)$ be the local homological Chern character map with respect to $\ZZ(W)^A \subset \M(W)^A \times \M(W)^A$ and let $\Td_{\M(W)^A} \in H^\bullet(\M(W)^A)$ be the Todd genus of $\M(W)^A$.
\begin{prop}
The \emph{bivariant Riemann-Roch map} 
\be \RR := \left ( 1 \boxtimes \Td_{\M(W)^A}\right ) \cup \ch_\bullet :K(\ZZ(W)^A)\rightarrow H_\bullet(\ZZ(W)^A, \Q) \ee
is compatible with convolution  in the sense that the following two diagrams commute.
$$ 
\begin{CD}
K(\ZZ(W)^A)^{\otimes 2} @>\star>> K(\ZZ(W)^A)\\
@V{\RR \otimes \RR}VV @VV{ \RR}V\\
H_\bullet(\ZZ(W)^A, \Q)^{ \otimes 2} @>\star>> H_\bullet(\ZZ(W)^A, \Q)
\end{CD}
\qquad\qquad
\begin{CD}
K(\ZZ(W)^A) \otimes K(\LL(W)^A)_\Q @>\star>> K(\LL(W)^A)_\Q\\
@V{\RR \otimes \ch_\bullet} VV @V{\cong}V{ \ch_\bullet}V\\
H_\bullet (\ZZ(W)^A, \Q) \otimes H_\bullet(\LL(W)^A, \Q) @>\star>> H_\bullet(\LL(W)^A, \Q)
\end{CD}
$$
\end{prop}
\begin{proof}
The commutativity of the first diagram is the statement of the bivariant Riemann-Roch theorem proven in \cite{CG}. That of the second follows by a similar argument, \emph{mutatis mutandis}. \cite{Nakajima} gives a proof that $\ch_\bullet : K(\LL(W)^A) \stackrel{\sim}{\rightarrow} H_\bullet(\LL(W)^A, \Q)$ is an isomorphism by factoring it through the rational Chow group $A_\bullet(\LL(W)^A, \Q) \cong K(\LL(W)^A)_\Q$ and establishing that the cycle map $A_\bullet(\LL(W)^A, \Q) \stackrel{\sim}{\longrightarrow} H_\bullet(\LL(W)^A, \Q)$ is an isomorphism.
\end{proof}

Combining all these results, we finally get
\begin{prop}
The following hold:
\begin{enumerate}
 \item[i.] there exists an algebra homomorphism $\Psi:\uepslg \rightarrow H_\bullet(\ZZ(W)^A, \C)$;
\item[ii.] ${\bf M}_{a} \cong H_\bullet(\LL(W)^A, \C)$ as $\uepslg$-modules.
\end{enumerate}
\end{prop}

\subsection{Generic standard modules}
\begin{defn}
\label{def:generic}
Given an $I$-graded $\C$-vector space $W$, we shall say that the semisimple element $a \in G(W) \times \C^*$ and, by extension, that the reductive abelian group $A$ it generates is \emph{generic} iff $\M_0(V, W)^A =  \{0\}$ for every $I$-graded vector $\C$-space $V$. We shall refer to the corresponding standard modules ${\bf M}_{a}$ as \emph{generic standard modules}.
\end{defn}
In type $\mathfrak a$, $\mathfrak d$, or $\mathfrak e$, proposition 4.2.2. in \cite{Nakajima} provides a sufficient condition for $a \in A$ to be generic;
\begin{prop}
\label{prop:suffgeneric}
If $\lambda /\lambda' \notin \varepsilon^{\Z^*}$ for every pair $\lambda, \lambda' \in \Sp(W)$, then $a \in A$ is generic.
\end{prop}
It is obvious that,
\begin{prop}
\label{prop:generic}
Whenever $a$ is generic, $\LL(V, W)^A \cong \M(V,W)^A$ is a nonsingular projective variety. 
\end{prop}

\begin{cor}\label{cor:nonsing}
If $\Sp(W)=\{\lambda_1, \dots, \lambda_l\}$ is a single $q$-segment of length $l$, the numbering being chosen in such a way that $\lambda_i<\lambda_{i+1}$ for all $i\in\range{1}{l}$, we let $W_i:=W(\lambda_i)$ denote the eigenspaces and $A_i:=A|_{W_i}$. Then, for every $l$-tuple of $I$-graded $\C$-vector spaces $(V_1, \dots, V_l)$, we have that $\LL(V_1, W_1)^{A_1} \times \cdots \times \LL(V_l,W_l)^{A_l}$ is a non-singular projective variety.
\end{cor}

\subsection{$\ell$-weight spaces as $\mathrm{U}^0_\varepsilon(\mathrm L\g)$-modules}
We first state three obvious lemmas.
\begin{lem} 
\label{lem:ra}
The following diagram commutes:
$$\begin{CD}
K^A(\M(W)) @>\delta_*>> K^A(\ZZ(W))\\
@V{i^*}VV @VV{r_a}V\\
K^A(\M(W)^A) @>\delta_*>> K^A(\ZZ(W)^A)
\end{CD}$$
\end{lem}
\begin{proof}
It follows from \cite{CG} that
$$\begin{CD}
K^A(\M(W)) @>\delta_*>> K^A(\ZZ(W))\\
@V{\lambda_A^{-1} \circ i^*}VV @VV{ (\lambda_A^{-1} \boxtimes \lambda_A^{-1}) \circ i^*}V\\
K^A(\M(W)^A) @>\delta_*>> K^A(\ZZ(W)^A)
\end{CD}$$
commutes. Observe furthermore that $p_1^*\lambda_A \otimes (\lambda_A^{-1} \boxtimes \lambda_A^{-1}) = 1 \boxtimes \lambda_A^{-1}$. We thus have
\be r_a \circ \delta_* = (1 \boxtimes \lambda_A^{-1} ) \circ i^* \circ \delta_* = p_1^* \lambda_A \otimes (\lambda_A^{-1} \boxtimes \lambda_A^{-1}) \circ i^* \circ \delta_*
= p_1^* \lambda_A \otimes \delta_* \circ \lambda_A^{-1} \circ i^*\,.\ee
Making use of the projection formula, we get
\be r_a \circ \delta_* = \delta_* (\delta^* \circ p_1^* \lambda_A \otimes \lambda_A^{-1} \circ i^*)\ee
and the result follows since $\delta^* \circ p_1^* = \id $.
\end{proof}

\begin{lem}
\label{lem:RR}
The following diagram commutes:
$$\begin{CD}
K(\M(W)^A) @>\delta_*>> K(\ZZ(W)^A)\\
@V{\ch_\bullet}VV @VV{\RR}V\\
H_\bullet(\M(W)^A) @>\delta_*>> H_\bullet(\ZZ(W)^A)
\end{CD} $$
\end{lem}
\begin{proof}
Observe first that $p_1^* \Td_{X^A}^{-1} \cdot ( \Td_{X^A} \boxtimes \Td_{X^A})= 1 \boxtimes \Td_{X^A}$. We thus have
\be \RR  \circ \delta_*=  p_1^* \Td_{X^A}^{-1} \cdot ( \Td_{X^A} \boxtimes \Td_{X^A}) \cup \ch_\bullet  \circ \delta_* \, .\ee
By the Riemann-Roch theorem, \cite{CG}, it follows that
\be \RR  \circ \delta_*=  p_1^* \Td_{X^A}^{-1} \cdot \delta_* (\Td_{X^A} \cup \ch_\bullet )\, .\ee
The projection formula can now be applied, yielding
\be \RR  \circ \delta_* = \delta_* (\delta^* \circ p_1^* \Td_{X^A}^{-1}  \cdot \Td_{X^A} \cup \ch_\bullet) \ee
and the result follows from $\delta^* \circ p_1^* = \id$.
\end{proof}

\begin{lem}
\label{lem:star}
The following diagram commutes:
$$\begin{CD}
H_\bullet(\M(W)^A) \otimes H_\bullet(\LL(W)^A) @>{\delta_* \otimes \id_{H_\bullet(\LL(W)^A)}}>> H_\bullet(\ZZ(W)^A) \otimes H_\bullet(\LL(W)^A)\\
@V\cap VV @VV{\star}V\\
H_\bullet(\LL(W)^A) @= H_\bullet(\LL(W)^A)
\end{CD} $$
\end{lem}
\begin{proof}
In the context of section \ref{Sec:Convolution}, we let $X_1=X_2= \M(W)^A$ and $X_3=$ pt, so that $X_1 \times X_2 \times X_3$ can be identified with $X_1 \times X_2$. Under this identification, we have $p_{13} = p_1$, $p_{23}=p_2$ and $p_{12} = \id_{X_1 \times X_2}$. It is clear then that, for all $u \in H_\bullet(\M(W)^A)$ and $v \in H_\bullet(\LL(W)^A)$,
\be \delta_* u \star v = p_{13 \, *} \left ( p_{12}^* \circ \delta_* u \cap p_{23}^* v \right ) = p_{1 \, *} \left (\delta_* u \cap p_2^* v \right ) = p_{1 \, *} \circ \delta_* \left ( u \cap \delta^* \circ p_2^* v \right )\, ,\ee
where we have used the projection formula. The result eventually follows from $p_1 \circ \delta = p_2 \circ \delta = \id_{\M(W)^A}$.
\end{proof}

\begin{rem}
Strictly speaking, the statement of the above lemma involves pulling back by the inclusion morphism $\LL(W)^A \hookrightarrow \M(W)^A$. However, we shall omit that pull back in order to simplify subsequent formulae.
\end{rem}

\begin{prop}
\label{prop:kaction}
For every $I$-graded $\C$-vector space $V$ and every $\rho \in \Hom(A, G(V))$, we have
\be \kk{\pm}{i}(z)   . v = \ch_\bullet \circ \ev_a \circ i_\rho^* \left ( q^{\rank(\F_i(V, W))}{\bigwedge}_{-1/z}^\pm (q^{-1}-q) \F_i(V, W)\right )\cap v\,,\ee
for all $v \in H_\bullet (\LL(V, W)[\rho], \mathbb C)$.
\end{prop}
\begin{proof}
It follows immediately from lemmas \ref{lem:ra}, \ref{lem:RR} and \ref{lem:star}.
\end{proof}
Setting
\be \label{lweight} k^\pm_i[\rho](z) := \ch_{\dim_\R \LL(V, W)^A} \circ \ev_a \circ i_\rho^*\left ( q^{\rank(\F_i(V, W))}{\bigwedge}_{-1/z}^\pm (q^{-1}-q) \F_i(V, W) \right ) \, ,\ee
for the $\ell$-weight associated with each $\rho \in \Hom(A, G(V))$, it follows that
\begin{cor}
For every $I$-graded $\C$-vector space $V$ and every $\rho \in \Hom(A, G(V))$,
\be H_\bullet (\LL(V, W)[\rho], \C) = \{v \in {\bf M}_{a} : \forall i \in I \, \exists N \in \N \, , \quad (\kk{\pm}{i}(z) - k^\pm_i[\rho](z))^N. v=0 \}\, . \ee
\end{cor}
\noindent As a consequence, $H_\bullet (\LL(V, W)[\rho], \C)$ gets identified with the $\ell$-weight space associated with the $\ell$-weight $k^\pm_i[\rho](z)$.

\subsection{$\varepsilon,t$-characters}
Under this identification, the highest $\ell$-weight of a given standard module ${\bf M}_a$, where $a=(s, \varepsilon) \in G(W)\times \C^*$ is semisimple, is obtained by taking $V= \{0\}$ and consequently $\rho=0$ in (\ref{lweight}), yielding
\be \chi_a  \left (q^{\rank(\W_i)} {\bigwedge}_{-1/z}^\pm (q^{-1}-q) q^{-1} \W_i\right )= \varepsilon^{\dim(W_i)} \frac{\chi_a({\bigwedge}_{-1/qz}^\pm q^{-1} \W_i)}{\chi_a({\bigwedge}_{-q/z}^\pm q^{-1} \W_i)}  =: \varepsilon^{\deg(P_{a,i})} \left (\frac{P_{a, i}(1/\varepsilon z)}{P_{a,i}(\varepsilon /z)} \right )^\pm\ee
where the arguments of $\chi_a$ should be thought of as $A$-modules and the last equality allows one to identify the Drinfel'd polynomial $\P_a= (P_{a,i})_{i \in I} \in \C[X]_1^{\rank(\g)}$ of ${\bf M}_a$ as
\be P_{a,i}(1/z) = \prod_{\lambda \in \Sp(W)} \left (1- \lambda(a)/\varepsilon z \right )^{\dim W_i(\lambda)}\, .\ee
Therefore, ${\bf M}_a = M(\P_a)$ in the notations of the introduction. 
\begin{rem}
It should be clear at this point that the entire set $\C[X]_1^{\rank(\g)}$ of Drinfel'd polynomials can be obtained by varying $a \in  G(W)\times \C^*$, while keeping $x=0$ fixed. This would actually be true for any $A$-fixed $x \in \M_0^{\mathrm{reg}}(V^0, W)$.
\end{rem}

Similarly, for a general $\ell$-weight, we have (\ref{lweight})
\be k^\pm_i[\rho](z) = \varepsilon^{\deg(Q_i) - \deg(R_i)} \frac{Q_i(1/\varepsilon z)}{R_i(1/\varepsilon z)} \, \frac{R_i(\varepsilon/z)}{Q_i(\varepsilon/z)}\, ,\ee
where
\bea\frac{Q_i(1/z)}{R_i(1/z)} &:=& P_i(1/z) \, \dfrac{\chi_a \left ( \displaystyle  \bigotimes_{\substack{h \in H\\ \inn(h) = i}}  {\bigwedge}^\pm_{-1/qz} \V_{\out(h)}   \right )}{\chi_a \left ( {\bigwedge}^\pm_{-1/z} \V_i \otimes {\bigwedge}^\pm_{-1/q^2 z} \V_i \right ) } \nn\\
&=& P_i(1/z) \, \dfrac{ \displaystyle \prod_{\substack{h \in H\\ \inn(h) = i}} \prod_{\lambda \in \Sp(\rho, V_{\out(h)})}(1- \lambda(a)/\varepsilon z )^{\dim V_{\out(h)}(\lambda)}}{\displaystyle  \prod_{\lambda \in \Sp(\rho, V_i)} ((1- \lambda(a)/z)(1- \lambda(a)/\varepsilon^2 z))^{\dim V_i(\lambda)}}\, .\eea
It is natural then to associate with every such $\ell$-weight, parametrized by the rational functions
\be \frac{Q_i(1/z)}{R_i(1/z)} = \frac{\prod_{k}(1- \alpha_{i,k}/z)}{\prod_{k}(1- \beta_{i,k}/z)} \, ,\ee
the monomial
\be m_\rho:=  \prod_{i \in I} \prod_{k} Y_{i, \alpha_{i,k}} Y^{-1}_{i, \beta_{i,k}} \ee
in the formal variables $(Y_{i, c})_{i \in I, c \in \C^*}$.

For every pair of $I$-garded $\C$-vector spaces $(V, W)$ and every $\rho \in \Hom(A, G(V))$, let $n := \dim_\C (\LL(V, W)[\rho])$ and let
\be{\mathscr P} (\LL(V, W)[\rho], t) := \sum_{k =0}^{2n} t^k \dim H_k(\LL(V, W)[\rho])  \ee
be the Poincar\'e polynomial of $\LL(V, W)[\rho]$. Following \cite{NakqtAxiom}, we make the natural
\begin{defn} For every ${\bf P} \in  \C[X]_1^{\rank(\g)}$, we define the $\varepsilon,t$-character of the standard module $M({\bf P})$ by
\be \chi_{\varepsilon,t}(M({\bf P})) := \sum_{[V]} \sum_{\rho \in \Hom(A, G(V))} {\mathscr P} (\LL(V, W)[\rho], t) \,\, m_\rho \in \Z[t][Y^{\pm 1}_{i, c}]_{i \in I, c \in \C^*}\,.\ee
\end{defn}

\begin{rem}
\label{rem:conceven}
The Poincar\'e polynomials ${\mathscr P} (\LL(V, W)[\rho], t)$ and, subsequently, the $\varepsilon,t$-characters are indeed polynomials in the variable $t^2$ as $H_\bullet(\LL(V, W)[\rho], \Z)$ is known to be concentrated in even degree \cite{Nakajima}.
\end{rem}

\begin{rem}
Since standard modules constitute a basis of the Grothendieck ring $\GrMod{\uepslg}$ of the category of finite dimensional $\uepslg$-modules, the above definition indeed gives a homomorphism of $\Z[t, t^{-1}]$-modules $\GrMod{\uepslg}\to \Z[t, t^{-1}][Y_{i, c}^{\pm1}]_{i \in I, c \in \C^*}$, see \cite{NakqtConj}.
\end{rem}

\section{Proof of the main theorem}
\label{sec:proof}
Let $a = (s, \varepsilon)$ denote a semisimple element of $G(W) \times \mathbb C^*$ such that $1 \notin \varepsilon^\Z$ and let $A$ be the Zariski closure of $a^{\mathbb Z}$. As usual, any $I$-graded $\mathbb C$-vector space $V$ is regarded as an $A$-module through some $\rho \in \Hom(A, G(V))$. Recall that $q \in \Hom (A, \C^*)$ is the composite $q: A \hookrightarrow G(W) \times \C^* \twoheadrightarrow \C^*$, where the second arrow is projection on the second factor.

\subsection{Single $q$-segment case} We first assume that $\Sp(W)$ is a single $q$-segment, in the terminology of section \ref{sec:qseg}. Since $\varepsilon$ is not a root of unity, $\Sp(W)$ is totally ordered. Similarly, for every $\rho \in \Hom(A, G(V))$, $\Sp(\rho, V)$ is a totally ordered finite $q$-segment. Let $\lambda_\rho:= \sup \Sp(\rho, V)$ 
and define
\be I_\rho:= \{i \in I : V_i(\lambda_\rho) \neq \{0\} \} \subseteq I \, .\ee
\begin{prop}
\label{prop:leadingeps}
For every $\rho \in \Hom(A, G(V))$, every $i \in I_\rho$ and every $m \in \mathbb N$, we have
\be i^*_\rho \F_i(V, W)(q^{m}\lambda_{\rho}) = \begin{cases} i^*_\rho\W_i(q \lambda_{\rho}) - i^*_\rho\V_i(\lambda_{\rho}) & \mbox{if } m=0;\\
 i^*_\rho\W_i(q^{1+m}\lambda_{\rho}) & \mbox{otherwise.}\end{cases}\ee
\end{prop}
\begin{proof}
It is immediate from the definitions of $i^*_\rho \F_i(V, W)(q^{m}\lambda_{\rho})$ and $\lambda_\rho$.
\end{proof}
Define
\be \xi_{i, \rho}(a):= c_1 \circ \ev_a \circ i_\rho^* \left ( (q-q^{-1})q^{\rank(\F_i(V, W))} \F_i(V, W)\right ) \in H^2(\M(V,W)[\rho], \C)\, .\ee
\begin{lem}\label{prop:leading}  We have :
\begin{enumerate}
\item[i.] for all $v \in H_k (\LL(V, W)[\rho], \mathbb C)$,
\be (\kk{+}{i,1} - k^+_{i,1}[\rho])  . v = \xi_{i, \rho}(a) \cap v \mod \bigoplus_{l \leq k-3}H_{l}(\LL(V, W)[\rho], \mathbb C)\,,\ee
\item[ii.] $\xi_{i, \rho}(a)/\lambda_\rho(a) \in H^2(\M(V,W)[\rho], \Z)[\varepsilon, \varepsilon^{-1}]$ and, indeed, setting $r := \rank (\F_i(V, W))$,
\be  \frac{\xi_{i, \rho}(a)}{\lambda_\rho(a)} =\varepsilon^{1+r}\,\, c_1\circ \det i_\rho^* \V_i(\lambda_\rho)^\vee
\mod \bigoplus_{k \leq r} \varepsilon^k  \, H^2(\M(V,W)[\rho], \Z) \,. \ee
\end{enumerate}
\end{lem}
\begin{proof}
\emph{i.} follows immediately from proposition \ref{prop:kaction} and \emph{ii.} follows by isolating in $\xi_{i,\rho}(a)$ the leading monomial in $\varepsilon$, making use of proposition \ref{prop:leadingeps}.
\end{proof}
Define $\mathfrak{sl}_2$ as the Lie algebra generated over $\C$ by the triple $(f,h,e)$ subject to the relations
\be [h,e]=2e\,, \qquad  [h,f]=-2f\,, \qquad  [e,f]=h \, , \label{eq:sl2}\ee
and let $\usl 2$ denote its universal enveloping algebra.
\begin{prop}
\label{Prop:sl2single}
If $\varepsilon$ is transcendental over $\Q$, then for every pair of $I$-graded $\C$-vector spaces $(V,W)$, every $\rho \in \Hom(A, G(V))$ and every $k \in I_\rho$, $H_\bullet (\LL(V, W)[\rho], \C)$ admits an $\usl 2$-module structure such that
\be f . v = \xi_{k, \rho}(a) \cap v\, , \qquad \qquad h . v = (m - \dim_\C \LL(V, W)^A) v\, ,\ee
for all $v \in H_{m}(\LL(V, W)[\rho], \C)$.
\end{prop}
\begin{proof}
Let $n:= \dim_\C \LL(V, W)^A$ and $b_m:= \dim_\C H_{m}(\LL(V, W)[\rho], \C)$ for all $m \in \range{0}{2n}$. We start by assuming that the single $q$-segment $\Sp(W)$ has length $1$. Then, by propositions \ref{prop:suffgeneric} and \ref{prop:generic}, $\LL(V, W)^A$ is a projective non-singular variety. It follows from lemma \ref{lem:Vample}.\emph{iv}. and the hard Lefschetz theorem that, for all $k \in I_\rho$ and every $m \in \range{0}{n}$,
\be \left (c_1 \circ \det i_\rho^* \V_k (\lambda_\rho)^\vee\right )^{\cap m} : H_{n+m}(\LL(V, W)[\rho], \Q) \stackrel{\sim}{\longrightarrow} H_{n - m}(\LL(V, W)[\rho], \Q) \label{map}\ee 
is an isomorphism. If, on the other hand, the $q$-segment $\Sp(W)=\{\lambda_1, \dots , \lambda_l\}$ has length $l>1$, we let  $W_i:=W(\lambda_i)$ denote the corresponding eigenspaces. Note that it suffices to consider the case
\be\LL(V,W)[\rho] \subset \M(V,W)[\rho] \label{eq:LstrictM}\ee
where the inclusion is strict; for otherwise, $\LL(V,W)[\rho] \cong \M(V,W)[\rho]$ is a nonsingular projective variety and (\ref{map}) again follows by the hard Lefschetz theorem. By proposition \ref{prop:BBdecomp}\emph{ii.}, the (possibly singular) projective variety $\LL(V,W)[\rho]$ admits a Bia\l ynicki-Birula type decomposition, given by eq. (\ref{Eq:LBBdecomp}), which is good. Thus, applying the generalized homology basis formula of \cite{CarrellGoresky}, we get
\bea  && H_\bullet(\LL(V, W)[\rho], \Q) = \label{eq:hdecomp}\\
&& \qquad  \bigoplus_{\substack{[V_1], \dots, [V_l]\\V \cong V_1 \oplus \cdots \oplus V_l}} \bigoplus_{\substack{\rho_1\in\Hom(A_1, G(V_1))\\ \vdots \\ \rho_l\in\Hom(A_l, G(V_l))}} H_{\bullet-2p(V_1, W_1; \cdots ; V_l, W_l; \rho_1, \dots,\rho_l)}(\LL(V_1, W_1)[\rho_1] \times \cdots \times \LL(V_l, W_l)[\rho_l] , \Q) \, ,\nn\eea
where we have set 
\be p(V_1, W_1; \cdots ; V_l, W_l; \rho_1, \dots, \rho_l):=\rank\left (T\M(V, W)[\rho]|_{\LL(V_1, W_1)[\rho_1] \times \cdots \times \LL(V_l, W_l)[\rho_l]}^+\right ) \, . \label{eq:rank}\ee
Note that $\LL(V, W)[\rho]$ being compact, the above formula holds for Borel-Moore homology as well as for ordinary homology. However, in the case of (\ref{eq:LstrictM}), $\LL(V, W)[\rho]$ is generally singular and the hard Lefschetz theorem does not necessarily apply. Let us therefore consider the middle perversity intersection homology of $\LL(V, W)[\rho]$, $IH_\bullet(\LL(V, W)[\rho], \Q)$. The latter also admits a homology basis formula as a result of \cite{Kirwan}, namely
\bea && IH_\bullet(\LL(V, W)[\rho], \Q) = \label{eq:ihdecomp} \\
&& \qquad \bigoplus_{\substack{[V_1], \dots, [V_l]\\V \cong V_1 \oplus \cdots \oplus V_l}} \bigoplus_{\substack{\rho_1\in\Hom(A_1, G(V_1))\\ \vdots \\ \rho_l\in\Hom(A_l, G(V_l))}} \HH^{-\bullet}(\LL(V_1, W_1)[\rho_1] \times \cdots \times \LL(V_l, W_l)[\rho_l], j_l^* i_l^! \IC_{\LL(V, W)[\rho]})\, , \nn \eea
where $\HH^\bullet$ denotes the hypercohomology functor, $\IC_X \in \mathcal D^b(X)$ is the intersection complex of sheaves associated with the variety $X$ and
\be i_l : \LL(V_1, W_1; \dots ; V_l, W_l)[\rho_1, \dots, \rho_l] \hookrightarrow \LL(V, W)[\rho]\ee
and
\be j_l:\LL(V_1, W_1)[\rho_1] \times \cdots \times \LL(V_l, W_l)[\rho_l]  \hookrightarrow \LL(V_1, W_1; \dots ; V_l, W_l)[\rho_1, \dots, \rho_l]\ee
denote the obvious inclusions. By proposition \ref{prop:BBdecomp}\emph{i.}, $\LL(V_1, W_1; \dots ; V_l, W_l)[\rho_1, \dots, \rho_l]$ is isomorphic to the total space of the restricted fibre bundle $T\M(V,W)[\rho]|_{\LL(V_1, W_1)[\rho_1] \times \cdots \times \LL(V_l, W_l)[\rho_l] }^+$ over the nonsingular basis $\LL(V_1, W_1)[\rho_1] \times \cdots \times \LL(V_l, W_l)[\rho_l]$ -- see corollary \ref{cor:nonsing}. Hence, $\LL(V_1, W_1)[\rho_1] \times \cdots \times \LL(V_l, W_l)[\rho_l]$ can be regarded as the zero section of $T\M(V,W)[\rho]|_{\LL(V_1, W_1)[\rho_1] \times \cdots \times \LL(V_l, W_l)[\rho_l] }^+$ and any inner product on the fibres of the latter allows one to define a tubular neighborhood of $\LL(V_1, W_1)[\rho_1] \times \cdots \times \LL(V_l, W_l)[\rho_l]$ in  $\LL(V_1, W_1; \dots ; V_l, W_l)[\rho_1, \dots, \rho_l]$. It follows  -- see \cite{GM}  -- that $j_l$ is a normally nonsingular inclusion and, consequently, that
\be j_l^*\IC_{\LL(V_1, W_1; \dots ; V_l, W_l)[\rho_1, \dots, \rho_l]} \cong \IC_{\LL(V_1, W_1)[\rho_1] \times \cdots \times \LL(V_l, W_l)[\rho_l]}[2p(V_1, W_1; \cdots ; V_l, W_l; \rho_1, \dots, \rho_l)]\,,\label{eq:jn}\ee
where the rightmost $[\,]$ denotes the shift functor over the derived category of bounded complexes of sheaves. Similarly, $\LL(V_1, W_1; \dots ; V_l, W_l)[\rho_1, \dots, \rho_l]$ admits a tubular neighborhood as a nonsingular subvariety in the nonsingular variety
\be \M(V, W)[\rho] \backslash \bigsqcup_{ \LL(V_1', W_1'; \dots ; V_l', W_l')[\rho_1', \dots, \rho_l']\neq \LL(V_1, W_1; \dots ; V_l, W_l)[\rho_1, \dots, \rho_l]}  \LL(V_1', W_1'; \dots ; V_l', W_l')[\rho_1', \dots, \rho_l']\,. \ee
In other words, there exists some open subset $V$ of the latter which is homeomorphic to some open neighborhood of the zero section of some vector bundle over $\LL(V_1, W_1; \dots ; V_l, W_l)[\rho_1, \dots, \rho_l]$. Now, clearly,
\be \LL(V_1, W_1; \dots ; V_l, W_l)[\rho_1, \dots, \rho_l] = V \cap \LL(V, W)[\rho]\, . \ee
It follows, \cite{GM},  that $i_l$ is a normally nonsingular inclusion as well and, consequently, that
\be i_l^!\IC_{\LL(V,W)[\rho]} \cong \IC_{ \LL(V_1, W_1; \dots ; V_l, W_l)[\rho_1, \dots, \rho_l]}\,.\label{eq:in}\ee
Substituting (\ref{eq:in}) and (\ref{eq:jn}) into eq. (\ref{eq:ihdecomp}) yields
\bea && IH_\bullet(\LL(V, W)[\rho], \Q) =\label{eq:ihdecompfin}\\
&&\qquad  \bigoplus_{\substack{[V_1], \dots, [V_l]\\V \cong V_1 \oplus \cdots \oplus V_l}} \bigoplus_{\substack{\rho_1\in\Hom(A_1, G(V_1))\\ \vdots \\ \rho_l\in\Hom(A_l, G(V_l))}} IH_{\bullet -2p(V_1, W_1; \cdots ; V_l, W_l; \rho_1, \dots, \rho_l)}(\LL(V_1, W_1)[\rho_1] \times \cdots \times \LL(V_l, W_l)[\rho_l],\Q)\, . \nn \eea
By corollary \ref{cor:nonsing}, $\LL(V_1, W_1)[\rho_1] \times \cdots \times \LL(V_l, W_l)[\rho_l]$ is nonsingular and it follows  -- see \eg \cite{KW}  -- that
\be IH_\bullet(\LL(V_1, W_1)[\rho_1] \times \cdots \times \LL(V_l, W_l)[\rho_l],\Q) = H_\bullet(\LL(V_1, W_1)[\rho_1] \times \cdots \times \LL(V_l, W_l)[\rho_l],\Q)\ee
for every $l$-tuple of $I$-graded $\C$-vector spaces such that $V\cong V_1\oplus \cdots \oplus V_l$ and every $l$-tuple of homomorphisms $(\rho_1, \dots , \rho_l) \in \Hom(A_1, G(V_1)) \times \cdots \times \Hom(A_l, G(V_l))$ such that $\rho = \rho_1 \oplus \cdots \oplus \rho_l$. Comparing (\ref{eq:hdecomp}) and (\ref{eq:ihdecompfin}) thus establishes that
\be H_\bullet(\LL(V, W)[\rho], \Q) = IH_\bullet(\LL(V, W)[\rho], \Q) \ee
and (\ref{map}) now follows in the general single $q$-segment case from the hard Lefschetz theorem for the intersection homology of the (possibly singular) projective variety $\LL(V, W)[\rho]$.

Therefore, viewing $(\xi_{k, \rho}(a))^{\cap m}$ as a map from $H_{n+m}(\LL(V, W)[\rho], \C)$ to  $H_{n - m}(\LL(V, W)[\rho], \C)$, it follows from lemma \ref{prop:leading}.\emph{ii.} that there exists a non-zero $P_m \in \Q[X, X^{-1}]$ such that
\be \det (\xi_{k, \rho}(a)^{\cap m}) = \lambda_\rho(a)^{b_{n+m}} P_m(\varepsilon) \, ,\ee
for every $m \in \range{0}{n}$ such that $n+m$ (and therefore $n-m$) be even -- see remark \ref{rem:conceven}. Since $\lambda_\rho(a) \in \C^*$, if $\det (\xi_{k, \rho}(a)^{\cap m})$ were zero, $\varepsilon$ would have to be a root of $P_m$, a contradiction. $\xi_{k, \rho}(a)^{\cap m}$ is thus an isomorphism for all $m \in \range{0}{n}$ as above.

Now let $T_0:= \range{1}{b_{2n}} = \{1\}$~\footnote{Remember that $\LL(V, W)[\rho]$ is connected.} and let $(v_{2n, t, 0})_{t \in T_0}$ be a non-zero vector in $H_{2n}(\LL(V, W)[\rho], \C)$. It is clear that, for each $r \in \range{0}{n}$, the
\be v_{2n, t, r} := (\xi_{k, \rho}(a)^{\cap r}) \cap v_{2n,t, 0}  \in H_{2(n-r)}(\LL(V, W)[\rho], \C)\, , \qquad t \in T_0\, ,\ee
are linearly independent for otherwise $\xi_{k, \rho}(a)^{\cap n}$ would fail to be an isomorphism, thus contradicting the conclusion of the previous paragraph. Let $n_1$ be the smallest positive integer in $\range{\lceil n/2 \rceil}{n}$ such that $(v_{2n, t, n-n_1-1})_{t \in T_0}$ be a basis of $H_{2n_1+2}(\LL(V, W)[\rho], \C)$. If $n_1>\lceil n/2 \rceil -1$, let $T_1:=\range{1}{b_{2n_1} - b_{2n}} = \range{1}{b_{2n_1} - 1}$ and complete $(v_{2n, t, n-n_1})_{t \in T_0}$ to a basis $(v_{2n, t, n-n_1})_{t \in T_0} \sqcup (v_{2n_1, t, 0})_{t \in T_1}$ of $H_{2n_1}(\LL(V, W)[\rho], \C)$. Repeating the above argument a finite number $S$ of times and setting $n_0:=n$, we get a basis $(v_{2n_s, t, r})_{s \in \range{0}{S}, t \in T_s, r \in \range{0}{n_s}}$ of $H_\bullet(\LL(V, W)[\rho], \C)$ in which 
\be \xi_{k, \rho}(a) \cap v_{2n_s, t, m} = v_{2n_s, t, m+1}\, .\ee
It is straightforward to check that setting
\be f.v_{2n_s, t, m} = v_{2n_s, t, m+1}\, , \quad h.v_{2n_s, t, m} = \left [2(n_s-m) -n\right ]v_{2n_s, t, m} \, ,  \quad e.v_{2n_s, t, m} = m(2n_s-n-m+1) v_{2n_s, t, m-1} \, , \nn \ee
provides the required $\usl 2$-module structure.
\end{proof}
\begin{rem}
In the last pragraph of the above proof, we have indeed constructed a decomposition of $H_\bullet(\LL(V, W)[\rho], \C)$ into simple $\mathrm{U}(\mathfrak{sl}_2)$-modules. This decomposition is of course isomorphic to the Lefschetz (primitive) decomposition, which is a classic corollary of the hard Lefschetz theorem -- see \eg \cite{Voisin}. Note that by the above proposition, letting $\mathscr P(\LL(V, W)[\rho],t)$ denote the Poincar\'e polynomial of $\LL(V, W)[\rho]$, $y^{-\dim_\C \LL(V, W)[\rho]}\mathscr P(\LL(V, W)[\rho],y)$ can now be regarded as the character associated with the $\mathrm{U}(\mathfrak{sl}_2)$-module $H_\bullet(\LL(V, W)[\rho], \C)$.
\end{rem}
In the particular case of a single $q$-segment, theorem \ref{thm:main} now follows from the complete reducibility of $H_\bullet(\LL(V, W)[\rho], \C)$ as an $\mathrm{U}(\mathfrak{sl}_2)$-module and from proposition \ref{prop:Jordan} and equation (\ref{filt1}) below. As observed in the introduction, proposition \ref{Prop:sl2single} and hence theorem \ref{thm:main} indeed hold for all but a finite number of algebraic values of $\varepsilon \in \C^*$, corresponding to the roots of the Laurent polynomials $P_m \in \Q[X, X^{-1}]$ appearing in the proof of proposition \ref{Prop:sl2single}.

\subsection{General case} 
We now assume that $\Sp(W)$ decomposes as in (\ref{qseg}). Endow $\mathrm{U}(\mathfrak{sl}_2)$ with the standard comultiplication $\Delta : \usl 2 \rightarrow \usl 2^{\otimes 2}$ defined by
\be \Delta (x) = x \otimes 1 + 1 \otimes x\, , \ee
for all $x \in {\mathfrak{sl}}_2$. With the notations of section \ref{sec:qseg}, we have
\begin{prop}
\label{prop:sl2gen}
If $\varepsilon$ is transcendental over $\Q$, then for every pair of $I$-graded $\C$-vector spaces $(V,W)$, every $\rho \in \Hom(A, G(V))$ and every $k \in I_\rho$, $H_\bullet (\LL(V, W)[\rho], \C)$ admits an $\usl 2$-module structure such that
\be H_\bullet (\LL(V, W)[\rho], \C) \cong  \bigotimes_{\Sigma \in \Sp(W)/\sim} H_\bullet (\LL(V_\Sigma, W_\Sigma)[\rho_\Sigma], \C) \, ,\ee
as $\usl 2$-modules; where, on the right-hand side, the $\usl 2$-module structure is induced, through the comultiplication $\Delta$, by that of proposition \ref{Prop:sl2single} on each tensor factor.
\end{prop}
\begin{proof}
This follows directly from propositions \ref{prop:qsegdecomp} and \ref{Prop:Comult} and from the existence of the K\"unneth isomorphisms
\be K^{G(W_1)\times \C^*}(\M(V_1, W_1)) \otimes K^{G(W_2)\times \C^*}(\M(V_2, W_2)) \stackrel{\sim}{\rightarrow} K^{G(W_1) \times G(W_2) \times \C^*}(\M(V_1, W_2)\times \M(V_2, W_2)) \,,\ee
\be K^{G(W_1)\times \C^*}(\LL(V_1, W_1)) \otimes K^{G(W_2)\times \C^*}(\LL(V_2, W_2)) \stackrel{\sim}{\rightarrow} K^{G(W_1) \times G(W_2) \times \C^*}(\LL(V_1, W_2)\times \LL(V_2, W_2)) \,,\ee
given by the corresponding external tensor products, for every quadruple of $I$-graded vector spaces $V_1$, $V_2$, $W_1$ and $W_2$, \cite{Nakajima, NakTens, VV}.
\end{proof}
The relation between the above $\mathrm{U}(\mathfrak{sl}_2)$-module structure and the Jordan filtration of $H_\bullet (\LL(V, W)[\rho], \C)$ by $\kk{+}{i,1} - k^+_{i,1}[\rho]$ arises from the following
\begin{prop}
\label{prop:Jordan}
Under the premises of proposition \ref{prop:sl2gen} -- and \emph{a fortiori} of proposition \ref{Prop:sl2single} --, the Jordan canonical forms of $\kk{+}{i,1} - k^+_{i,1}[\rho]$ and $\xi_{i, \rho}(a)$ -- and hence of $f$ -- over $H_\bullet(\LL(V, W)[\rho], \C)$ are equivalent, for every $i \in I_\rho$.
\end{prop}
\begin{proof}
Let $n:= \dim_\C \LL(V, W)^A$. Since $H_\bullet(\LL(V, W)[\rho], \C)$ is an $\mathrm{U}(\mathfrak{sl}_2)$-module by proposition \ref{prop:sl2gen}, it admits a decomposition into simple $\mathrm{U}(\mathfrak{sl}_2)$-modules;
\be H_\bullet(\LL(V, W)[\rho], \C) \cong \bigoplus_{s \in S} V_s^{\oplus m(s)} \, ,\ee
where $S\subseteq \{2p-n:p \in \range{\lceil n/2 \rceil}{n} \}$ is a finite set of dominant weights of $\mathfrak{sl}_2$ and $m(s):= [H_\bullet(\LL(V, W)[\rho], \C) : V_s] \in \N^*$ is the multiplicity of the simple $\mathrm{U}(\mathfrak{sl}_2)$-module $V_s$ with highest weight $s$ as a summand in $H_\bullet(\LL(V, W)[\rho], \C)$. For every $s \in S$ and every $t \in \range{1}{m(s)}$, let $v_{s, t, 0}$ be a highest weight vector in the $t$-th copy of $V_s$ and set
\be v_{s,t,r}:= \xi_{i, \rho}(a)^{\cap r} \cap v_{s,t, 0} \, ,\ee
for every $r \in \range{0}{s}$. Then, $(v_{s,t,r})_{s \in S, t \in \range{1}{m(s)}, r \in \range{0}{s}}$ clearly constitutes a Jordan basis for $\xi_{i, \rho}(a)$ over $H_\bullet(\LL(V, W)[\rho], \C)$.

Now, assume that there exists a set of vectors $(w_{s,t,0})_{s \in S, t \in \range{1}{m(s)}}$ in $H_\bullet(\LL(V, W)[\rho], \C)$ such that
\be\label{condw} w_{s,t,0} = v_{s,t,0} \mod \bigoplus_{k < s+n} H_{k}(\LL(V, W)[\rho], \C) \qquad \mbox{and} \qquad (\kk{+}{i,1} - k^+_{i,1}[\rho])^{s+1}.w_{s,t,0} =0 \, ,\ee
for every $s \in S$ and $t \in \range{1}{m(s)}$. Under this assumption, define
\be\label{defwr} w_{s,t,r} := (\kk{+}{i,1} - k^+_{i,1}[\rho])^{r}.w_{s,t,0}  \, , \ee
for all $r \in \range{0}{s}$. It follows from lemma \ref{prop:leading}.\emph{i.} that
\be w_{s,t,r} = v_{s,t,r} \mod \bigoplus_{k < s+n-2r} H_{k}(\LL(V, W)[\rho], \C)\, .\ee
Since $(v_{s,t,r})_{s \in S, t \in \range{1}{m(s)}, r \in \range{0}{s}}$ is a basis of $H_{k}(\LL(V, W)[\rho], \C)$, so is $(w_{s,t,r})_{s \in S, t \in \range{1}{m(s)}, r \in \range{0}{s}}$. In the latter, by construction, $\kk{+}{i,1} - k^+_{i,1}[\rho]$ is in Jordan canonical form and the result clearly follows.

So it suffices to construct a set $(w_{s,t,0})_{s \in S, t \in \range{1}{m(s)}}$ satisfying (\ref{condw}) as above. To do so, we consider the dominant weights of $S$ in (natural) decreasing order and let first 
\be w_{n, t, 0} := v_{n, t, 0}\,,\ee
for every $t \in \range{1}{m(n)}$. The conditions of (\ref{condw}) clearly hold. Define $(w_{n,t,r})_{t \in \range{1}{m(n)}, r \in \range{0}{n}}$ according to (\ref{defwr}). If it constitutes a basis of $H_\bullet(\LL(V, W)[\rho], \C)$, we are done. Otherwise, $S \backslash \{n\} \neq \emptyset$ and we let $s_1$ denote its maximum (in the natural order). Observe that $(w_{n,t,r})_{t \in \range{1}{m(n)}, r \in \range{(n+s_1)/2+p}{n}}$ is a basis of $\bigoplus_{k \leq n-s_1-2p} H_{k}(\LL(V, W)[\rho], \C)$ for every $p \in \range{1}{(n-s_1)/2}$. Furthermore, since $\xi_{i, \rho}(a)^{\cap s_1+1} \cap v_{s_1,t, 0}=0$, it follows from lemma \ref{prop:leading}.\emph{i.} that
\be (\kk{+}{i,1} - k^+_{i,1}[\rho])^{s_1+1}.v_{s_1,t,0} \in \bigoplus_{k \leq n-s_1-4} H_{k}(\LL(V, W)[\rho], \C) \ee
and therefore that
\be (\kk{+}{i,1} - k^+_{i,1}[\rho])^{s_1+1}.v_{s_1,t,0} = \sum_{\substack{t \in \range{1}{m(n)}\\ r \in \range{(n+s_1)/2+2}{n}}}  a_{t,r} w_{n,t,r}\ee
for some tuple of complex numbers $(a_{t,r})_{t \in \range{1}{m(n)}, r \in \range{(n+s_1)/2+2}{n}}$. Now clearly, the
\be w_{s_1, t, 0} := v_{s_1, t, 0} - \sum_{\substack{t \in \range{1}{m(n)}\\ r \in \range{(n+s_1)/2+2}{n}}}  a_{t,r} w_{n,t,r-s_1-1}\, ,\ee
with $t \in \range{1}{m(s_1)}$, satisfy the conditions of (\ref{condw}). By repeating the above argument a finite number of times for all the weights in $S$, we construct the required set $(w_{s,t,0})_{s \in S, t \in \range{1}{m(s)}}$, which concludes the proof.
\end{proof}
Recall the Jordan filtration (\ref{jordanfiltr}). In the case at hand, this reads $F_\bullet M({\bf P})_\rho = F_\bullet \left (H_\bullet(\LL(V, W)[\rho], \C) \right )$. It is clear from propositions \ref{prop:sl2gen} and \ref{prop:Jordan} that, under the same hypotheses,
\be\label{filt1} F_k M({\bf P})_\rho = M({\bf P})_\rho \cap \ker \left(\kk{+}{i,1} - k^+_{i,1}[\rho]\right )^{k+1} \, ,\ee
for every $\rho \in \Hom(A, G(V))$, every $i \in I_\rho$ and every $k \in \N$. Hence, there exists a $\sigma \in S_{1+\dim_\C \LL(V, W)[\rho]}$ such that the positive integers
\be\dim\left ( F_k M({\bf P})_\rho /F_{k-1} M({\bf P})_\rho \right )= \dim H_{2\sigma(k)}(\LL(V, W)[\rho], \C) \ee
constitute a weakly decreasing sequence as $k$ increases from $0$ to $n_\rho:=\dim_\C \LL(V, W)[\rho]$. It is readily checked that $$\sigma_\rho :  \range{0}{n_\rho}  \ni k \mapsto \begin{cases} 
                                                                                                                 \lfloor n_\rho/2 \rfloor - k/2 \quad  \mbox{for even $k$}\\
														 \lfloor n_\rho/2 \rfloor + \lceil k/2 \rceil  \quad \mbox{for odd $k$}
                                                                                                                \end{cases}
$$ is such a permutation. Theorem \ref{thm:main} follows.

\section{Examples}
We now give a few examples of applications of theorem \ref{thm:main}. Throughout this section, $\varepsilon \in \Cx$ is transcendental over $\Q$. In order to simplify expressions and to save space, we shall write monomials in $\varepsilon,t$-characters with the convention that each factor of $Y_{i,a\varepsilon^n}^m$, where $i \in I$, $a \in \Cx$ and $m,n \in \Z$, is represented as $i_{n}^m $. We shall also present $\varepsilon,t$-characters as graphs whenever this is possible, so as to make the overall weight structure apparent.

 \subsection{Thick fundamental modules}
Fundamental modules being thin in type $\mathfrak a$, we shall look for thick fundamental modules over quantum affine algebras of types $\mathfrak d$ and $\mathfrak e$.
\subsubsection{Second fundamental module of $\uepsh{d}{4}$}
The simplest example arises in type $\mathfrak d_4$. Assuming the following labelling for the nodes of the Dynkin diagram of type $\mathfrak d_4$
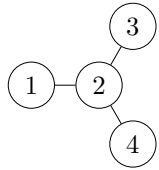
\begin{center}
\btp
\node[draw, circle] (1) at (0,0) {$1$};
\node[draw, circle] (2) at (1,0) {$2$};
\node[draw, circle] (3) at (1.5,0.87) {$3$};
\node[draw, circle] (4) at (1.5,-0.87) {$4$};
\draw (1) -- (2);
\draw (2) -- (3);
\draw (2) -- (4);
\etp
\end{center}
the $\varepsilon,t$-character of the second fundamental representation $V_{2,a}$ of $\uepsh{d}{4}$ can be depicted as follows
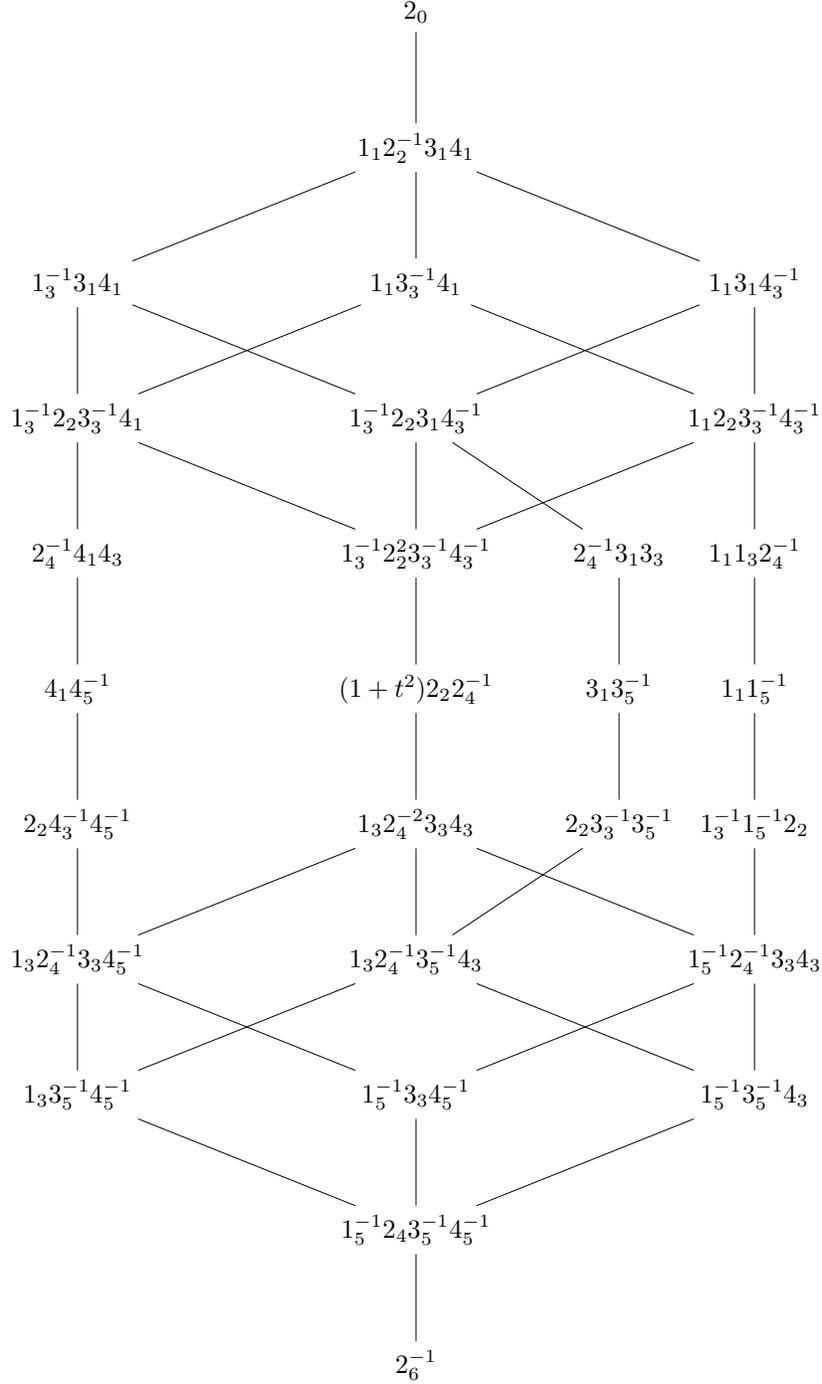
\begin{center}
\btp
\node (1) at (0,0) {$2_0$};
\node (2) at (0,-2) {$1_1 2_2^{-1} 3_1 4_1$};
\node (3) at (-5,-4) {$1_3^{-1} 3_1 4_1$};
\node (4) at (0,-4) {$1_1 3_3^{-1} 4_1$};
\node (5) at (5,-4) {$1_1 3_1 4_3^{-1}$};
\node (6) at (-5,-6) {$1_3^{-1} 2_2 3_3^{-1} 4_1$};
\node (7) at (0, -6) {$1_3^{-1} 2_2 3_1 4_3^{-1}$};
\node (8) at (5,-6) {$1_1 2_2 3_3^{-1} 4_3^{-1}$};
\node (9) at (-5,-8) {$2_4^{-1} 4_1 4_3$};
\node (10) at (0,-8) {$1_3^{-1} 2_2^2 3_3^{-1} 4_3^{-1}$};
\node (11) at (3, -8) {$2_4^{-1} 3_1 3_3$};
\node (12) at ( 5, -8) {$1_1 1_3 2_4^{-1}$};
\node (13) at (-5,-10) {$4_1 4_5^{-1}$};
\node (14) at (0,-10) {$(1 + t^2)2_2 2_4^{-1}$};
\node(15)  at (3,-10) {$3_1 3_5^{-1}$};
\node (16) at (5,-10) {$1_1 1_5^{-1}$};
\node (17) at (-5,-12) {$2_2 4_3^{-1} 4_5^{-1}$};
\node (18) at (0, -12) {$1_3 2_4^{-2} 3_3 4_3$};
\node (19) at (3, -12) {$2_2 3_3^{-1} 3_5^{-1}$};
\node (20) at (5, -12) {$1_3^{-1} 1_5^{-1} 2_2$};
\node (21) at (-5, -14) {$1_3 2_4^{-1} 3_3 4_5^{-1}$};
\node (22) at (5,-14) {$1_5^{-1} 2_4^{-1} 3_3 4_3$};
\node(23) at (0,-14) {$1_3 2_4^{-1} 3_5^{-1} 4_3$};
\node (24) at (0,-16) {$1_5^{-1} 3_3 4_5^{-1}$};
\node (25) at (5, -16) {$1_5^{-1} 3_5^{-1} 4_3$};
\node (26) at (-5, -16) {$1_3 3_5^{-1} 4_5^{-1}$};
\node (27) at (0, -18) {$1_5^{-1} 2_4 3_5^{-1} 4_5^{-1}$};
\node (28) at (0,-20) {$2_6^{-1}$};
\draw (1) -- (2); 
\draw (2) -- (3); 
\draw (2) -- (4);
\draw (2) -- (5);
\draw (3) -- (6);
\draw (3) -- (7);
\draw (4) -- (6);
\draw (5) -- (7);
\draw (4) -- (8);
\draw (5) -- (8);
\draw (6) -- (9);
\draw (6) -- (10);
\draw (7) -- (10);
\draw (8) -- (10);
\draw (7) -- (11);
\draw (8) -- (12);
\draw (9) -- (13);
\draw (10) -- (14);
\draw (11) -- (15);
\draw (12) -- (16);
\draw (13) -- (17);
\draw (14) -- (18);
\draw (15) -- (19);
\draw (16) -- (20);
\draw(17) -- (21);
\draw(18) -- (21);
\draw (18) -- (22);
\draw (20) -- (22);
\draw (18) -- (23);
\draw (19) -- (23);
\draw (21) -- (24);
\draw (22) -- (24);
\draw (22) -- (25);
\draw (23) -- (25);
\draw (21) -- (26);
\draw (23) -- (26);
\draw (24) -- (27);
\draw (25) -- (27);
\draw (26) -- (27);
\draw (27) -- (28);
\etp
\end{center}
The monomial $2_{2}2_{4}^{-1}$ has prefactor $1+t^2$ from which theorem \ref{thm:main} implies, $V_{2,a}$ being a standard module, that the corresponding $2$-dimensional $\ell$-weight space has a single Jordan chain of length $2$ rather than two Jordan chains of length one which was the only other possibilty. Note that this result fits in well with the intuition one might get from the type $\mathfrak a_1$ result by regarding the subgraph
\begin{center}
\btp
\node (10) at (0,0){$1_3^{-1} 2_2^2 3_3^{-1} 4_3^{-1}$};
\node (14) at (4,0){$(1 + t^2)2_2 2_4^{-1}$};
\node (18) at (8, 0) {$1_3 2_4^{-2} 3_3 4_3$};
\draw (10) -- (14);
\draw (14) -- (18);
\etp
\end{center}
with factors of $Y_{i\neq 2, a}^{\pm 1}$ deleted, as the $\varepsilon,t$-character of the thick standard $\uepsslt$-module $V_{a\varepsilon^2}^{\otimes 2}$. Indeed, for that particular module, the result of theorem \ref{thm:main} admits an alternative proof which consists in decomposing $V_{2,a}$ into modules over the Dynkin diagram subalgebra $\uepsslt^{(2)}$ of $\uepsh{d}{4}$ and invoking the type $\mathfrak a_1$ result of \cite{YZ}. 

\subsubsection{Thick fundamental modules of $\uepsh{e}{6}$}
We label the nodes of the Dynkin diagram of $\mathfrak e_6$ as follows
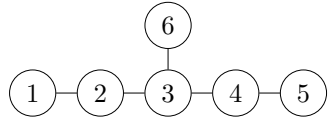
\begin{center}
\btp
\node[draw, circle] (1) at (0,0) {$1$};
\node[draw, circle] (2) at (1,0) {$2$};
\node[draw, circle] (3) at (2,0) {$3$};
\node[draw, circle] (4) at (3,0) {$4$};
\node[draw, circle] (5) at (4,0) {$5$};
\node[draw, circle] (6) at (2, 1) {$6$};
\draw (1) -- (2);
\draw (2) -- (3);
\draw (3) -- (4);
\draw (4) -- (5);
\draw (3) -- (6);
\etp
\end{center}
In the $\varepsilon,t$-character of the third fundamental module $V_{3,a}$ of $\uepsh{e}{6}$, the monomial
$2_5 2_7^{-1} 4_5 4_7^{-1} 6_5 6_7^{-1}$ appears with prefactor $1+3t^2+3t^4+t^6$. 
According to theorem \ref{thm:main}, the corresponding $8$-dimensional $\ell$-weight space therefore decomposes into a Jordan chain of length $4$ and two Jordan chains of lengths $2$. Again, this fits in well with what one might expect from the type $\mathfrak a_1$ result. Indeed, the graph depicting the $\varepsilon,t$-character of $V_{3,a}$ contains a subgraph
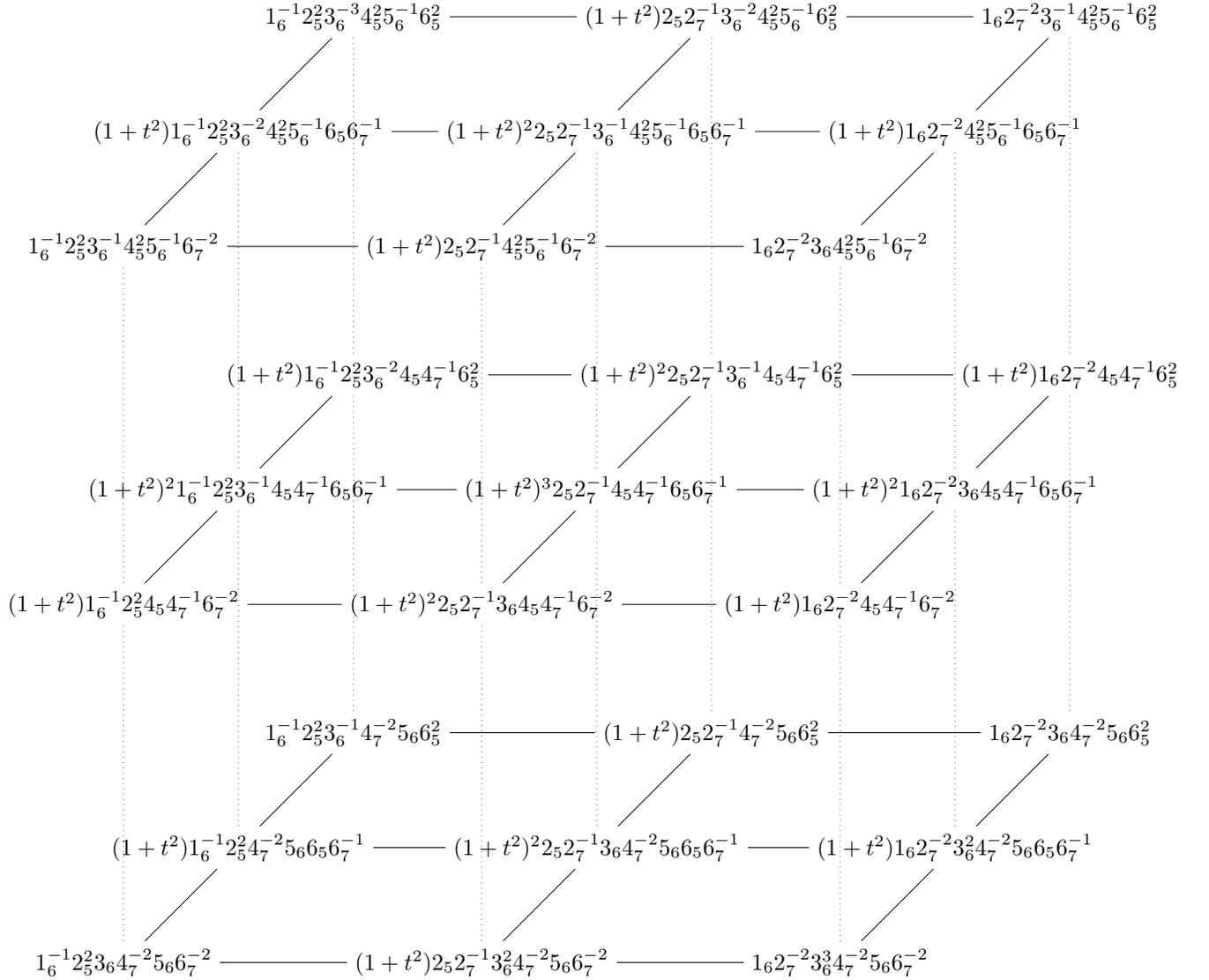
\begin{center}
\btp
\node (1) at (-6,0,0) {$1_6^{-1}2_5^23_6^{-3}4_5^25_6^{-1}6_5^2$};
\node (2) at (0,0,0) {$(1+t^2)2_52_7^{-1}3_6^{-2}4_5^25_6^{-1}6_5^2$};
\node (3) at (6,0,0) {$1_62_7^{-2}3_6^{-1}4_5^25_6^{-1}6_5^2$};
\node (4) at (-6,0,5) {$(1+t^2)1_6^{-1}2_5^23_6^{-2}4_5^25_6^{-1}6_56_7^{-1}$};
\node (5) at (0,0,5) {$(1+t^2)^22_52_7^{-1}3_6^{-1}4_5^25_6^{-1}6_56_7^{-1}$};
\node (6) at (6,0,5) {$(1+t^2)1_62_7^{-2}4_5^25_6^{-1}6_56_7^{-1}$};
\node (7) at (-6,0,10) {$1_6^{-1}2_5^23_6^{-1}4_5^25_6^{-1}6_7^{-2}$};
\node (8) at (0,0,10) {$(1+t^2)2_52_7^{-1}4_5^25_6^{-1}6_7^{-2}$};
\node (9) at (6,0,10) {$1_62_7^{-2}3_64_5^25_6^{-1}6_7^{-2}$};
\node (10) at (-6,-6,0) {$(1+t^2)1_6^{-1}2_5^23_6^{-2}4_54_7^{-1}6_5^2$};
\node (11) at (0, -6,0) {$(1+t^2)^22_52_7^{-1}3_6^{-1}4_54_7^{-1}6_5^2$};
\node (12) at ( 6,-6,0) {$(1+t^2)1_62_7^{-2}4_54_7^{-1}6_5^2$};
\node (13) at (-6,-6,5) {$(1+t^2)^21_6^{-1}2_5^23_6^{-1}4_54_7^{-1}6_56_7^{-1}$};
\node (14) at (0, -6,5) {$(1+t^2)^32_52_7^{-1}4_54_7^{-1}6_56_7^{-1}$};
\node (15) at ( 6,-6,5) {$(1+t^2)^21_62_7^{-2}3_64_54_7^{-1}6_56_7^{-1}$};
\node (16) at (-6,-6,10) {$(1+t^2)1_6^{-1}2_5^24_54_7^{-1}6_7^{-2}$};
\node (17) at (0, -6,10) {$(1+t^2)^22_52_7^{-1}3_64_54_7^{-1}6_7^{-2}$};
\node (18) at ( 6,-6,10) {$(1+t^2)1_62_7^{-2}4_54_7^{-1}6_7^{-2}$};
\node (19) at (-6,-12,0) {$1_6^{-1}2_5^23_6^{-1}4_7^{-2}5_66_5^2$};
\node (20) at (0, -12,0) {$(1+t^2)2_52_7^{-1}4_7^{-2}5_66_5^2$};
\node (21) at ( 6,-12,0) {$1_62_7^{-2}3_64_7^{-2}5_66_5^2$};
\node (22) at (-6,-12,5) {$(1+t^2)1_6^{-1}2_5^24_7^{-2}5_66_56_7^{-1}$};
\node (23) at (0, -12,5) {$(1+t^2)^22_52_7^{-1}3_64_7^{-2}5_66_56_7^{-1}$};
\node (24) at ( 6,-12,5) {$(1+t^2)1_62_7^{-2}3_6^24_7^{-2}5_66_56_7^{-1}$};
\node (25) at (-6,-12,10) {$1_6^{-1}2_5^23_64_7^{-2}5_66_7^{-2}$};
\node (26) at (0, -12,10) {$(1+t^2)2_52_7^{-1}3_6^{2}4_7^{-2}5_66_7^{-2}$};
\node (27) at ( 6,-12,10) {$1_62_7^{-2}3_6^34_7^{-2}5_66_7^{-2}$};
\draw (1) -- (2); 
\draw[dotted] (1) -- (10);
\draw (2) -- (3); 
\draw[dotted] (2) -- (11);
\draw (1) -- (4);
\draw (4) -- (5);
\draw[dotted] (4) -- (13);
\draw (5) -- (6);
\draw (2) -- (5);
\draw (3) -- (6);
\draw[dotted] (3) -- (12);
\draw (4) -- (7);
\draw (5) -- (8);
\draw[dotted] (5) -- (14);
\draw (6) -- (9);
\draw[dotted] (6) -- (15);
\draw (7) -- (8);
\draw[dotted] (7) -- (16);
\draw (8) -- (9);
\draw[dotted] (8) -- (17);
\draw[dotted] (9) -- (18);
\draw (10) -- (11);
\draw (10) -- (13);
\draw[dotted] (10) -- (19);
\draw (11) -- (12);
\draw[dotted] (11) -- (20);
\draw (11) -- (14);
\draw (12) -- (15);
\draw[dotted] (12) -- (21);
\draw (13) -- (14);
\draw (13) -- (16);
\draw[dotted] (13) -- (22);
\draw (14) -- (15);
\draw (14) -- (17);
\draw[dotted] (14) -- (23);
\draw (15) -- (18);
\draw[dotted] (15) -- (24);
\draw (16) -- (17);
\draw[dotted] (16) -- (25);
\draw(17) -- (18);
\draw[dotted] (17) -- (26);
\draw[dotted] (18) -- (27);
\draw (20) -- (21);
\draw (20) --(23);
\draw (19) -- (20);
\draw (19) -- (22);
\draw (22) -- (23);
\draw (22) -- (25);
\draw (23) -- (24);
\draw (21) -- (24);
\draw (23) -- (26);
\draw (24) -- (27);
\draw (25) -- (26);
\draw (26) -- (27);
\etp
\end{center}
which, after deleting factors of $Y_{i\notin \{2,4,6\},a}^{\pm 1}$, can be interpreted as the $\varepsilon,t$-character of the standard module $V_{2,a\varepsilon^5}^{\otimes 2} \otimes V_{4, a\varepsilon^5}^{\otimes 2} \otimes V_{6, a\varepsilon^5}^{\otimes 2}$ over the diagram subalgebra of $\uepsh{e}{6}$ with Dynkin diagram
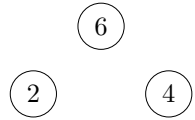
\begin{center}
\btp
\node[draw, circle] (2) at (1,0) {$2$};
\node[draw, circle] (4) at (3,0) {$4$};
\node[draw, circle] (6) at (2, 1) {$6$};
\etp
\end{center}
Therefore, in the particular case of the $\ell$-weight spaces associated with the monomials appearing in the previous subgraph, theorem \ref{thm:main} admits an alternative proof which consists in decomposing $V_{3,a}$ into modules over that diagram subalgebra and invoking the type $\mathfrak a_1$ result of \cite{YZ}. It should be stressed however that not all cases can be dealt with by restricting to diagram subalgebras and invoking the type $\mathfrak a_1$ result, as the next two examples illustrate.

First in the same $\varepsilon,t$-character as above, the monomial $3_4 3_8^{-1}$ appears with prefactor $1 + 4t^2 + t^4$. 
Therefore, by theorem \ref{thm:main}, the corresponding $6$-dimensional $\ell$-weight space decomposes into a Jordan chain of length $3$ and three Jordan chains of lengths $1$. However, $y^{-2}+4+y^2$ is not the character of any simple $\usl 2$-module nor of any tensor product thereof. Consequently, one should not expect in that case to find an alternative proof of the result of theorem \ref{thm:main} relying only on the type $\mathfrak a_1$ result, in the spirit of the previous two examples. To be more specific, the monomial $3_4 3_8^{-1}$ appears in a subgraph of the form
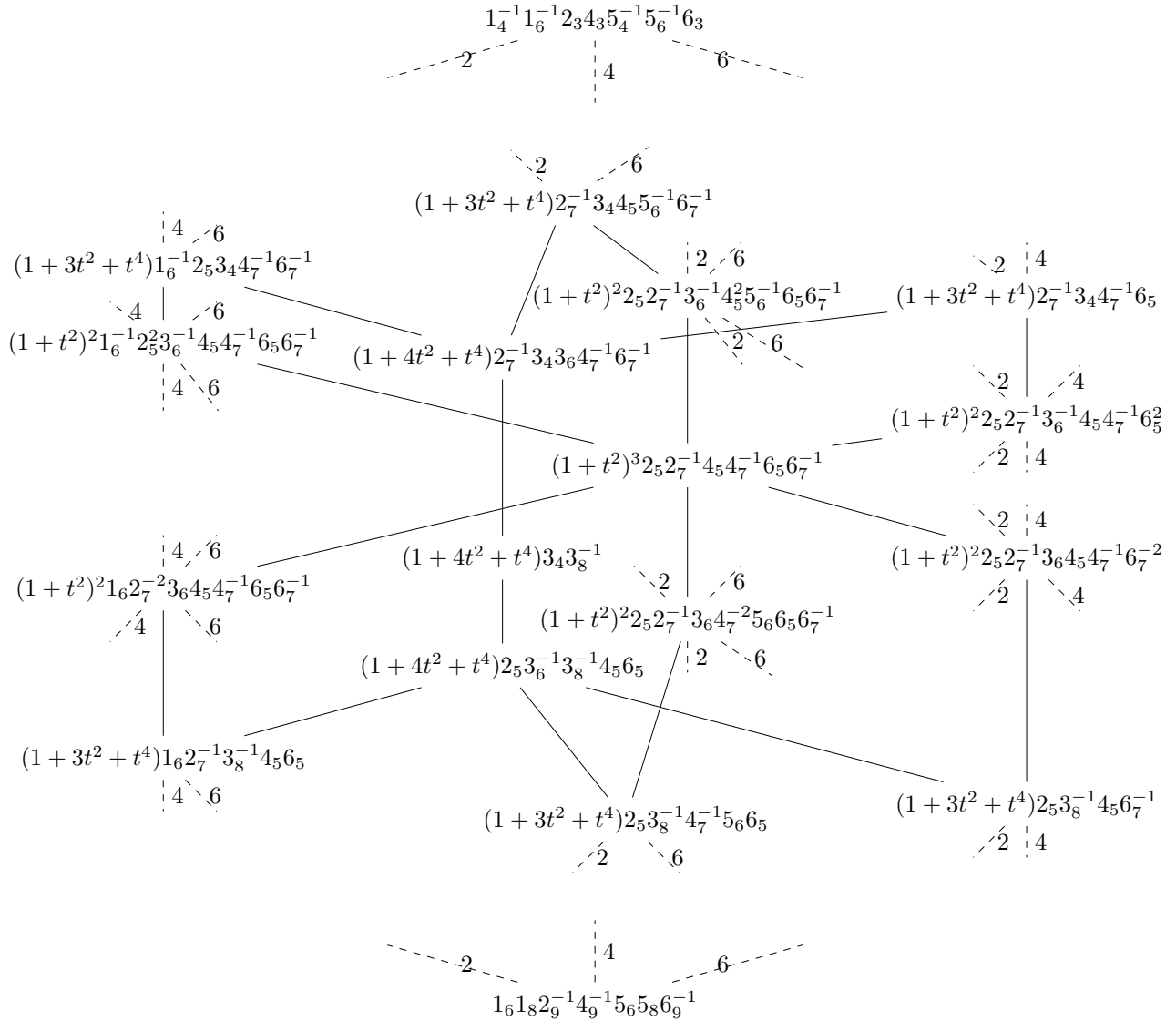
\begin{center}
\btp
\node (1) at (-1.5,-.75) {$(1+4t^2+t^4)3_43_8^{-1}$};
\node(2) at (-1.5,2.5) {$(1+4t^2+t^4)2_7^{-1}3_43_64_7^{-1}6_7^{-1}$};
\node(3) at (-7,4) {$(1+3t^2+t^4)1_6^{-1}2_53_44_7^{-1}6_7^{-1}$};
\node (4) at (-.5,5) {$(1+3t^2+t^4)2_7^{-1}3_44_55_6^{-1}6_7^{-1}$};
\node (5) at (7, 3.5) {$(1+3t^2+t^4)2_7^{-1}3_44_7^{-1}6_5$};
\node (6) at (-7, 2.75) {$(1+t^2)^21_6^{-1}2_5^23_6^{-1}4_54_7^{-1}6_56_7^{-1}$};
\node (7) at (1.5, .75) {$(1+t^2)^32_52_7^{-1}4_54_7^{-1} 6_56_7^{-1}$};
\node (8) at (7,-.75) {$(1+t^2)^22_52_7^{-1}3_64_54_7^{-1}6_7^{-2}$};
\node (9) at (-7, -1.25) {$(1+t^2)^21_62_7^{-2}3_64_54_7^{-1}6_56_7^{-1}$};
\node (10) at (1.5, 3.5) {$(1+t^2)^22_52_7^{-1}3_6^{-1}4_5^25_6^{-1}6_56_7^{-1}$};
\node (11) at (7, 1.5) {$(1+t^2)^22_52_7^{-1}3_6^{-1}4_54_7^{-1}6_5^2$};
\node (12) at (1.5,-1.75) {$(1+t^2)^22_52_7^{-1}3_64_7^{-2}5_66_56_7^{-1}$};
\node(top) at (0,8) {$1_4^{-1} 1_6^{-1} 2_3 4_3 5_4^{-1} 5_6^{-1} 6_3$};
\node (not1) at (-3.5, 7) {};
\node (not2) at (0,6.5) {};
\node (not3) at (3.5,7) {};
\node (not4) at (-7,5) {};
\node (not5) at (-6,4.75) {};
\node (not6) at (-1.5, 6) {};
\node (not6p) at (1, 6) {};
\node (not7) at (6, 4.25) {};
\node (not7p) at (7, 4.5) {};
\node (not8) at (-7, 1.5) {};
\node (not8p) at (-6, 1.5) {};
\node (not8n) at (-8, 3.5) {};
\node (not8nn) at (-6, 3.5) {};
\node (not9) at (6, .5) {};
\node (not9p) at (7, .5) {};
\node (not9n) at (6, 2.5) {};
\node (not9nn) at (8, 2.5) {};
\node (not10) at (-7, -.25) {};
\node (not10p) at (-6, -.25) {};
\node (not11) at (6, .25) {};
\node (not11p) at (7, .25) {};
\node (not12) at (-8, -2.25) {};
\node (not12p) at (-6, -2.25) {};
\node (not13) at (6, -1.75) {};
\node (not13p) at (8, -1.75) {};
\node(not14) at (.5,-.75) {};
\node(not14p) at (2.5,-.75) {};
\node(not14n) at (1.5,-2.75) {};
\node(not14nn) at (3,-2.75) {};
\node(not15) at (1.5,4.5) {};
\node(not15p) at (2.5,4.5) {};
\node(not15n) at (2.5,2.25) {};
\node(not15nn) at (3.5,2.25) {};
\node (nnot1) at (-3.5, -7) {};
\node (nnot2) at (0,-6.5) {};
\node (nnot3) at (3.5,-7) {};
\node (nnot2n) at (-7, -5) {};
\node (nnot2nn) at (-6, -5) {};
\node (nnot3n) at (-.5, -6) {};
\node (nnot3nn) at (1.5, -6) {};
\node (nnot4n) at (6, -5.75) {};
\node (nnot4nn) at (7, -5.75) {};
\node (-1) at (-1.5, -2.5) {$(1+4t^2+t^4)2_5 3_6^{-1} 3_8^{-1} 4_5 6_5$};
\node (-2) at (-7,-4) {$(1+3t^2+t^4) 1_6 2_7^{-1} 3_8^{-1} 4_5 6_5$};
\node (-3) at (.5,-5) {$(1+3t^2+t^4)2_5 3_8^{-1} 4_7^{-1} 5_6 6_5$};
\node (-4) at (7,-4.75) {$(1+3t^2+t^4)2_5 3_8^{-1} 4_5 6_7^{-1}$};
\node (bot) at (0, -8) {$1_6 1_8 2_9^{-1} 4_9^{-1} 5_6 5_8 6_9^{-1}$};
\draw[dashed] (top) -- (not1) node[midway, right] {$2$}; 
\draw[dashed] (top) -- (not2) node[midway, right] {$4$}; 
\draw[dashed] (top) -- (not3) node[midway, left] {$6$}; 
\draw[dashed] (bot) -- (nnot1) node[midway, right] {$2$}; 
\draw[dashed] (bot) -- (nnot2) node[midway, right] {$4$}; 
\draw[dashed] (bot) -- (nnot3) node[midway, left] {$6$}; 
\draw (1) -- (2);
\draw (3) -- (2);
\draw[dashed] (-2) -- (nnot2n) node[midway, right] {$4$};
\draw[dashed] (-2) -- (nnot2nn) node[midway, right] {$6$};
\draw[dashed] (-3) -- (nnot3n) node[midway, right] {$2$};
\draw[dashed] (-3) -- (nnot3nn) node[midway, right] {$6$};
\draw[dashed] (-4) -- (nnot4n) node[midway, right] {$2$};
\draw[dashed] (-4) -- (nnot4nn) node[midway, right] {$4$};
\draw[dashed] (3) -- (not4) node[midway, right] {$4$};
\draw[dashed] (3) -- (not5) node[midway, right] {$6$};
\draw[dashed] (4) -- (not6) node[midway, right] {$2$};
\draw[dashed] (4) -- (not6p) node[midway, right] {$6$};
\draw[dashed] (5) -- (not7) node[midway, right] {$2$};
\draw[dashed] (5) -- (not7p) node[midway, right] {$4$};
\draw[dashed] (6) -- (not8) node[midway, right] {$4$};
\draw[dashed] (6) -- (not8p) node[midway, right] {$6$};
\draw[dashed] (6) -- (not8n) node[midway, right] {$4$};
\draw[dashed] (6) -- (not8nn) node[midway, right] {$6$};
\draw[dashed] (11) -- (not9) node[midway, right] {$2$};
\draw[dashed] (11) -- (not9p) node[midway, right] {$4$};
\draw[dashed] (11) -- (not9n) node[midway, right] {$2$};
\draw[dashed] (11) -- (not9nn) node[midway, right] {$4$};
\draw[dashed] (12) -- (not14) node[midway, right] {$2$};
\draw[dashed] (12) -- (not14p) node[midway, right] {$6$};
\draw[dashed] (12) -- (not14n) node[midway, right] {$2$};
\draw[dashed] (12) -- (not14nn) node[midway, right] {$6$};
\draw[dashed] (10) -- (not15) node[midway, right] {$2$};
\draw[dashed] (10) -- (not15p) node[midway, right] {$6$};
\draw[dashed] (10) -- (not15n) node[midway, right] {$2$};
\draw[dashed] (10) -- (not15nn) node[midway, right] {$6$};
\draw[dashed] (9) -- (not10) node[midway, right] {$4$};
\draw[dashed] (9) -- (not10p) node[midway, right] {$6$};
\draw[dashed] (9) -- (not12) node[midway, right] {$4$};
\draw[dashed] (9) -- (not12p) node[midway, right] {$6$};
\draw[dashed] (8) -- (not11) node[midway, right] {$2$};
\draw[dashed] (8) -- (not11p) node[midway, right] {$4$};
\draw[dashed] (8) -- (not13) node[midway, right] {$2$};
\draw[dashed] (8) -- (not13p) node[midway, right] {$4$};
\draw (4) -- (10);
\draw (3) -- (6);
\draw (6) -- (7);
\draw (4) --(2);
\draw (5) -- (2);
\draw (7) -- (8);
\draw (7) -- (10);
\draw (7) --(11);
\draw (7) -- (12);
\draw (-3) -- (12);
\draw (5) -- (11);
\draw (8) -- (-4);
\draw (7) -- (9);
\draw (9) -- (-2);
\draw (1) -- (-1);
\draw (-1) -- (-2);
\draw (-1) --(-3);
\draw (-1) -- (-4);
\etp
\end{center}
of the complete graph depicting the $\varepsilon,t$-character of $V_{3,a}$. Although the above subgraph with factors of $Y_{i\notin\{2,3,4,6\}, a}^{\pm 1}$ deleted turns out to depict the $\varepsilon,t$-character of the standard module $V_{2,3}\otimes V_{4,3}\otimes V_{6,3}$ over the $\uepsh{d}{4}$ diagram subalgebra\footnote{Inherited labelling understood for the nodes of the type $\mathfrak d_4$ Dynkin diagram.} of $\uepsh{e}{6}$, it is impossible to conclude anything about the Jordan block structure of the $\ell$-weight space associated with the corresponding monomial from the sole type $\mathfrak a_1$ result. The next example provides a more basic illustration of that fact.

\subsection{Standard modules over $\uepsh{a}{2}$}
In type $\mathfrak a_2$, 
the $\varepsilon,t$-character of the generic standard module $V_{1,a} \otimes V_{1,a}$ can be depicted as follows
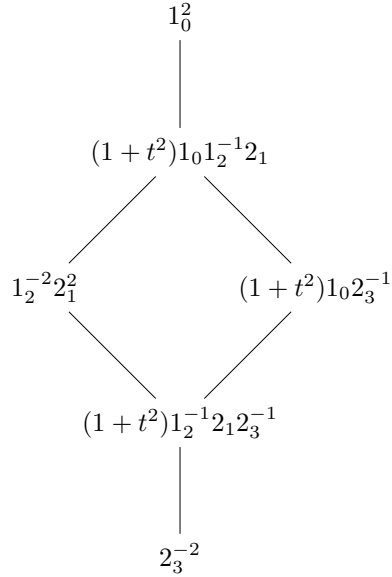
\begin{center}
\btp
\node (1) at (0,3) {$1_0^2$};
\node(2) at (0,1) {$(1+t^2)1_01_2^{-1}2_1$};
\node(3) at (-2,-1) {$1_2^{-2} 2_1^2$};
\node(4) at(2,-1) {$(1 + t^2)1_0 2_3^{-1}$};
\node (5) at (0,-3) {$(1 + t^2)1_2^{-1} 2_12_3^{-1}$};
\node (6) at (0, -5) {$2_3^{-2}$};
\draw (1) -- (2);
\draw (2) -- (3);
\draw (2) -- (4);
\draw (3) -- (5);
\draw (4) -- (5);
\draw (5) -- (6);
\etp
\end{center}
The monomials $1_01_2^{-1}2_1$, $1_0 2_3^{-1}$ and $1_2^{-1} 2_12_3^{-1}$ appear with prefactor $1+t^2$. According to theorem \ref{thm:main}, the corresponding $2$-dimensional $\ell$-weight spaces consist of single Jordan chains of length two. In view of the type $\mathfrak a_1$ result, this is expected for $1_01_2^{-1}2_1$ and $1_2^{-1} 2_12_3^{-1}$ since
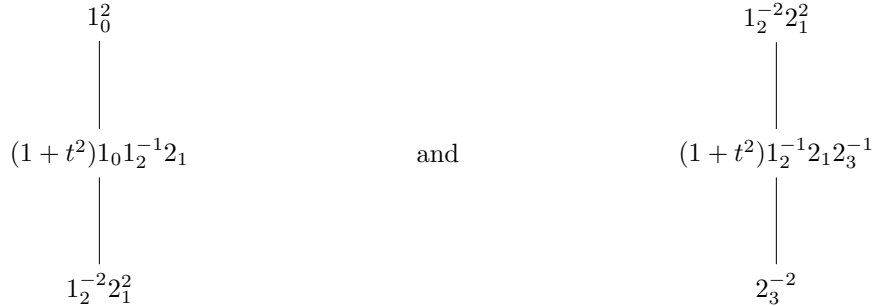
\begin{center}
\btp
\node (1) at (-5,2) {$1_0^2$};
\node(2) at (-5,0) {$(1+t^2)1_01_2^{-1}2_1$};
\node(3) at (-5,-2) {$1_2^{-2} 2_1^2$};
\node(4) at (5,2) {$1_2^{-2} 2_1^2$};
\node (5) at (5,0) {$(1 + t^2)1_2^{-1} 2_12_3^{-1}$};
\node (6) at (5, -2) {$2_3^{-2}$};
\node (7) at (0,0) {and};
\draw (1) -- (2);
\draw (2) -- (3);
\draw (4) -- (5);
\draw (5) -- (6);
\etp
\end{center}
with respectively $Y_{2,a\varepsilon}$ and $Y_{1, a\varepsilon^2}^{-1}$ deleted can be regarded as the $\varepsilon,t$-characters of the thick standard $\uepsslt$-modules $V_{a}^{\otimes 2}$ and $V_{a\varepsilon}^{\otimes 2}$ respectively. Therefore, for those particular $\ell$-weights, the result of theorem \ref{thm:main} admits another proof which consists in decomposing $V_{1,a} \otimes V_{1,a}$ into modules over the Dynkin diagram subalgebras of $\uepsh{a}{2}$ and invoking the type $\mathfrak a_1$ result of \cite{YZ}.  However, a similar reasoning does not obviously hold for $1_0 2_3^{-1}$ and it is worth emphasizing that the formation of a non-trivial Jordan block in the corresponding $\ell$-weight space is a typically higher rank phenomenon.


Let us eventually consider the $\varepsilon,t$-character of the standard module $V_{1,a} \otimes V_{2, a\varepsilon}$. It can be depicted as follows
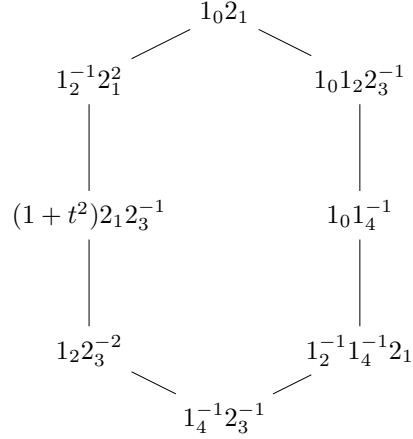
\begin{center}
\btp
\node (1) at (0,0) {$1_02_1$};
\node(2) at (-2,-1) {$1_2^{-1}2_1^2$};
\node(3) at (2,-1) {$1_01_22_3^{-1}$};
\node(4) at(-2,-3) {$(1 + t^2)2_12_3^{-1}$};
\node (5) at (2,-3) {$1_01_4^{-1}$};
\node (6) at (-2, -5) {$1_22_3^{-2}$};
\node (7) at (2,-5){$1_2^{-1}1_4^{-1}2_1$};
\node (8) at (0,-6){$1_4^{-1}2_3^{-1}$};
\draw (1) -- (2);
\draw (1) -- (3);
\draw (2) -- (4);
\draw (3) -- (5);
\draw (4) -- (6);
\draw (5) -- (7);
\draw (6) -- (8);
\draw (7) -- (8);
\etp
\end{center}
The monomial $2_12_3^{-1}$ appears with prefactor $(1 + t^2)$. According to theorem \ref{thm:main}, the corresponding $2$-dimensional $\ell$-weight space consists of a single Jordan chain of length two. This is expected since the subdiagram
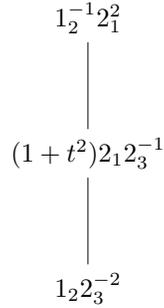
\begin{center}
\btp
\node(2) at (-2,2) {$1_2^{-1}2_1^2$};
\node(4) at(-2,0) {$(1 + t^2)2_12_3^{-1}$};
\node (6) at (-2, -2) {$1_22_3^{-2}$};
\draw (2) -- (4);
\draw (4) -- (6);
\etp
\end{center}
with $Y_{1, a\varepsilon^2}^{\pm 1}$ deleted can be regarded as the $\varepsilon,t$-character of the thick standard $\uepsslt$-module $V_{a\varepsilon}^{\otimes 2}$. Hence, for that particular module, theorem \ref{thm:main} admits an alternative proof which consists in decomposing $V_{1,a} \otimes V_{2,a\varepsilon}$ into modules over the Dynkin diagram subalgebras of $\uepsh{a}{2}$ and invoking the type $\mathfrak a_1$ result of \cite{YZ}. It should be stressed however that the existence of non-trivial Jordan chains in the tensor product of two fundamental modules forming a non-trivial $q$-segment -- here $\{a, aq\}$ -- is a purely higher rank phenomenon.

\newpage

\begin{thebibliography}{Nak01b}

\bibitem[BB73]{BB}
A.~Bia\l ynicki-Birula, \emph{{Some theorems on actions of algebraic groups}},
Ann. of Math. \textbf{98} (1973), 2nd series, no.~3, 480--497.

\bibitem[BBCM]{BBCMG}
A.~Bia\l ynicki-Birula, J.~B.~Carrell and W.~M.~McGovern, \emph{{Algebraic quotients. Torus Actions and cohomology. The adjoint representation and the adjoint action.}}, in Encyclopedia of Mathematical Sciences, vol. 131, Springer-Verlag (2002).

\bibitem[CG83]{CarrellGoresky}
J.~B.~Carrell and R.~M.~Goresky, \emph{{A decomposition theorem for the integral homology of a variety}},
Invent. math. \textbf{73}, (1983), 367 -- 381.

\bibitem[CG]{CG}
N.~Chriss and V.~Ginzburg, \emph{{Representation theory and complex geometry}},
  Modern Birkhauser Classics, Birkhauser, (1997), 495 p.

\bibitem[CP]{CPbook}
V.~Chari and A.~Pressley, \emph{{A guide to quantum groups}}, Cambridge, UK:
  Univ. Pr. (1994) 651 p.

\bibitem[CP91]{CPsl2}
\bysame, \emph{Quantum affine algebras}, Comm. Math. Phys. \textbf{142} (1991),
  261--283.

\bibitem[FM01]{FM}
E.~Frenkel and E.~Mukhin, \emph{{Combinatorics of q-characters of
  finite-dimensional representations of quantum affine algebras}}, Commun.
  Math. Phys. \textbf{216} (2001), 23--57.

\bibitem[FR98]{FR}
E.~Frenkel and N.~Reshetikhin, \emph{{The q-characters of representations of
  quantum affine algebras and deformations of W-algebras}}, Contemp. Math.
  \textbf{248} ({1998}), 163--205.

\bibitem[GM83]{GM}
M.~Goresky and R.~MacPherson, \emph{{Intersection homology II}}, Invent. Math. \textbf{71}, 77 -- 129.

\bibitem[GV93]{GV}
V.~Ginzburg and E.~Vasserot, \emph{{Langlands reciprocity for affine quantum
  groups of type $A_n$}}, Int. Math. Res. Notices ({1993}), no.~3, 67--85.

\bibitem[Kin94]{King}
A.~D.~King, \emph{{Moduli of representations of finite dimensional algebras}},
  Quart. J. Math. \textbf{45} ({1994}), 515--530.

\bibitem[Kir88]{Kirwan}
F.~Kirwan, \emph{{Intersection homology and torus actions}},
 J. of the Am. Math. Soc. \textbf{1} (1988), n. 2, 385 -- 400.

\bibitem[KW]{KW}
F.~Kirwan and J.~Woolf, \emph{{An introduction to intersection homology theory}}, 2nd ed., Taylor \& Francis group (2006).

\bibitem[Kni95]{Knight}
H.~Knight, \emph{{Spectra of Tensor Products of Finite Dimensional
  Representations of Yangians}}, Journal of Algebra \textbf{174} (1995), no.~1,
  187 -- 196.

\bibitem[Lun73]{Luna}
D.~Luna, \emph{{Slices \'etal\'es}}, M\'emoires de la Soci\'et\'e
  Math\'ematique de France \textbf{33} ({1973}), 81--105.

\bibitem[MFK]{Mumford}
D.~Mumford, J.~Fogarty, and F.~Kirwan, \emph{{Geometric Invariant Theory}},
  Ergebnisse der Mathematik und ihrer Grenzgebiete 34, Springer, Third Enlarged
  Edition (2002), 292 p.

\bibitem[Nak98]{NakKM}
H.~Nakajima, \emph{{Quiver varieties and Ka\v c-Moody algebras}}, Duke Math.
  J.. \textbf{91} ({1998}), no.~3, 515--560.

\bibitem[Nak01a]{Nakajima}
\bysame, \emph{{Quiver varieties and finite dimensional representations of
  quantum affine algebras }}, J. Amer. Math. Soc. \textbf{14} ({2001}),
  145--238.

\bibitem[Nak01b]{NakTens}
\bysame, \emph{{Quiver varieties and tensor products}}, Invent. Math.
  \textbf{146} ({2001}), 399--449.

\bibitem[Nak01c]{NakqtConj}
\bysame, \emph{{$t$-analogs of the $q$-characters of finite dimensional
  representations of quantum affine algebras}}, Physics and Combinatorics,
  Proceedings of the Nagoya 2000 International Workshop (2001), 195--218.

\bibitem[Nak04]{NakqtAxiom}
\bysame, \emph{{Quiver varieties and $t$-analogs of $q$-characters of quantum
  affine algebras}}, Annals of mathematics \textbf{160} ({2004}), no.~3,
  1057--1097.

\bibitem[Ser58]{Serre}
J.-P. Serre, \emph{{Espaces fibr\'es alg\'ebriques}}, S\'eminaire Claude
  Chevalley \textbf{3} ({1958}), no.~1, 1--37.

\bibitem[Tho87]{Thomason}
R.~Thomason, \emph{{Algebraic K-theory of group scheme actions, in Algebraic
  topology and algebraic K-theory}}, Ann. of Math. Studies \textbf{113}
  ({1987}), 539--563.

\bibitem[Vas98]{Vass}
E.~Vasserot, \emph{{Affine quantum groups and equivariant K-theory}},
  Transformation Groups \textbf{3} ({1998}), no.~3, 269--299.

\bibitem[Voi]{Voisin}
C.~Voisin, \emph{{Hodge theory and complex algebraic geometry, vol. 1}},
  Cambridge Studies in Advanced Mathematics, 76.

\bibitem[VV02]{VV}
M.~Varagnolo and E.~Vasserot, \emph{{Standard modules of quantum affine
  algebras}}, Duke Math. J. \textbf{111} ({2002}), no.~3, 509--533.

\bibitem[YZ12]{YZ}
C.~A.~S. Young and R.~Zegers, \emph{{On $q, t$-characters and the $\ell$-weight
  Jordan Filtration of Standard
  $\mathrm{U}_q(\widehat{\mathfrak{sl}}_2)$-modules }}, Int. Math. Res. Notices ({2012}), 2012 (10), 2179--2211.

\end{thebibliography}

\providecommand{\bysame}{\leavevmode\hbox to3em{\hrulefill}\thinspace}
\providecommand{\MR}{\relax\ifhmode\unskip\space\fi MR }
\providecommand{\MRhref}[2]{%
  \href{http://www.ams.org/mathscinet-getitem?mr=#1}{#2}
}
\providecommand{\href}[2]{#2}

\end{document}